\documentclass[12pt,reqno]{amsart}
\usepackage{amsmath,amssymb,amsthm}
\usepackage[mathscr]{eucal}
\def\R{\mathbb R}
\def\C{\mathbb C}

\def\F{\mathscr F}
\def\PP{\mathbb P}
\def\CP{\mathbb{CP}}

\def\CC{\mathscr C} 

\def\restr{\negthickspace \mid}
\def\&{\wedge}
\def\stranspose#1{{}^t\negthinspace{#1}}
\def\itranspose#1{\,{}^t\negthinspace\negthinspace{#1}}
\def\transpose#1{{}^t\negthinspace{#1}}

\def\cod{\operatorname{cod}}

\def\calS{\mathcal S}
\def\Q{\mathcal Q}
\def\utwo{\mathfrak{u}(2)}
\def\so{\mathfrak{so}}
\def\calP{\mathcal P}
\def\V{\mathcal V}

\def\II{\operatorname{II}}
\def\rJ{\mathsf J} 
\def\rK{\mathsf K}

\def\bigJ{\widehat J}
\def\SH{{\hat S}} 

\def\k{\delta}

\def\w{\omega}

\def\pit{\widetilde\pi}
\def\phit{\widetilde\phi}

\def\I{\mathcal I}
\newcommand{\J}{{\EuScript J}}

\def\ri{\mathrm i}
\def\tmod{\ \text{mod}\, }

\def\Jt{\widetilde J}

\newtheorem{theorem}{Theorem}
\newtheorem{lemma}[theorem]{Lemma}
\newtheorem{cor}[theorem]{Corollary}
\newtheorem{prop}[theorem]{Proposition}
\newtheorem*{theorem*}{Theorem}
\theoremstyle{definition}

\theoremstyle{remark}

\begin{document}
\title[Austere Submanifolds of Dimension Four]{Austere Submanifolds of Dimension Four: \\ Examples and Maximal Types}
\author{Marianty Ionel}
\address{Dept. of Mathematics, University of Toledo, 2801 W. Bancroft St., Toledo OH 43606
}
\email{mionel@utnet.utoledo.edu}
\author{Thomas Ivey}
\address{Dept. of Mathematics, College of Charleston, 66 George St., Charleston SC 29424}
\email{iveyt@cofc.edu}


\begin{abstract}
Austere submanifolds in Euclidean space were introduced by
Harvey and Lawson in connection with their study of calibrated geometries.
The algebraic possibilities for second fundamental forms of 4-dimensional
austere submanifolds were classified by Bryant, into three types which
we label A, B, and C.  In this paper, we show that type A submanifolds correspond
exactly to real K\"ahler submanifolds, we construct new examples of such submanifolds
in $\R^6$ and $\R^{10}$, and we obtain classification results on submanifolds
with second fundamental forms of maximal type.
\end{abstract}

\date{\today}
\maketitle

\section{Introduction}

\subsection*{Definitions and Background}
Recall that for an immersed submanifold $M^n \subset \R^{n+r}$ with
normal bundle $N(M)$, the second fundamental form $\II: TM \otimes
TM \to N(M)$ is defined by
$$\II(X,Y) = \pi_N \nabla_X Y,$$
where $X,Y$ are tangent vectors to $M$, $\nabla$ is the Euclidean
connection in $\R^{n+r}$, and $\pi_N$ is the orthogonal projection
onto the normal bundle. Then $M$ is {\em austere} if, for any normal
vector field $\nu$, the eigenvalues of the quadratic form
$\II_\nu(X,Y):= \nu \cdot \II(X,Y)$ with respect to the metric are
at each point symmetrically arranged around zero on the real line
(equivalently, all odd degree symmetric polynomials in these
eigenvalues vanish). When $n=2$, this just means that $M$ is a
minimal surface in $\R^{2+r}$. However, when $n>2$ the austere
condition is stronger than minimality,
 and leads to a highly overdetermined system of PDEs for the immersion.
For example, because of nonlinearity of the higher degree symmetric
polynomials, it does not suffice to impose the eigenvalue condition
on $\II_{\nu}$ when $\nu$ runs over a basis $\{\nu_a\}$, $1\le a\le
r$, for the orthogonal complement of $T_p M$; rather, the eigenvalue
condition applies to all quadratic forms in the space $|\II_p|
\subset S^2 T_p^*M$ spanned by $\nu_a \cdot \II$.

The austere condition was introduced by Harvey and Lawson \cite{HL}
in connection with special Lagrangian submanifolds.  A special
Lagrangian submanifold in $\C^n$ is a submanifold of real dimension
$n$ that is both Lagrangian and minimal. The importance of special
Lagrangian submanifolds lies mainly in the fact that they are
area-minimizing.  (These submanifolds have also received much recent
attention because of their relation to mirror symmetry \cite{SYZ}.)
Harvey and Lawson showed that the conormal bundle of an immersed
submanifold $M\subset \R^{n}$ is special Lagrangian in the cotangent
bundle $T^*{\R^n}$, equipped with its canonical symplectic structure
and metric, if and only if $M$ is an austere submanifold. Moreover,
Karigiannis and Min-Oo  \cite{KM} showed that a similar result holds
for submanifolds of the $n$-sphere; namely, the conormal bundle of
$M^k\subset S^n$ is special Lagrangian in $T^*S^n$, equipped with
the Stenzel metric and symplectic structure \cite{St} which make
$T^*{S^n}$ a Calabi-Yau manifold, if and only if $M$ is an austere
submanifold of $S^n$.

A systematic study and classification of austere submanifolds of
dimension $3$ in Euclidean space was first undertaken by Bryant
\cite{Baustere}, and generalized by Dajczer and Florit \cite{DF} to
austere submanifolds of arbitrary dimension whose Gauss map has rank
two. In studying the austere submanifolds of Euclidean space, we are
led first to an algebraic problem. Taking $V=\R^n$ with the standard
inner product, a linear subspace $\calS\subset S^2 V^*$ is called
{\em austere} if any element of $\calS$ has eigenvalues occurring in
oppositely signed pairs. Since any subspace of an austere subspace
is also austere, and any isometry of $V$ carries one austere
subspace to another, to classify austere subspaces it suffices to
find all {\em maximal} austere subspaces of $S^2 V^*$ up to
isometries. Bryant classified these spaces for $n=3$ and $n=4$. For
$n=3$, he also described the austere $3$-folds while the case $n=4$
was left open. We recall below his results.

\begin{theorem}[Bryant]  Let $V=\R^3$ and let $\calS \subset S^2 V^*$ be a maximal
austere subspace.  Then $\calS$ is $O(3)$-conjugate to one of the
following:
\begin{align*}
(a) \qquad \calS_A &= \left\{\left.
\begin{bmatrix} A & 0 \\ 0 & 0 \end{bmatrix}
\right| A \text{ is a traceless symmetric $2\times 2$ matrix}
\right\},
\\
(b) \qquad \calS_B &= \left\{\left.
\begin{bmatrix} 0 & b \\ \transpose{b} & 0 \end{bmatrix}
\right| b \text{ is a $1\times 2$ row vector}\right\}.
\end{align*}
Moreover, if $M^3\subset \R^n$ is an austere submanifold such that
for every $p$ in an open subset of $M$ the span $|\II_p|$ is
two-dimensional and is $O(3)$-conjugate to $\calS_A$ then $M$ is a
product of a minimal surface in $\R^{n-1}$ with a line, an
 open subset of a cone or a {\em twisted cone} over a minimal surface in the sphere $S^{n-1}$.
Likewise, if $|\II_p|$ is conjugate to $\calS_B$, then $M$ is an
open subset of a {\em generalized helicoid}.
\end{theorem}

In this context, a generalized helicoid $M^n \subset \R^{2n-1}$ is
the image of a parametrization
$$(x_0,\ldots,x_{n-1})\mapsto (x_0, x_1 \cos(\lambda_1 x_0),x_1 \sin(\lambda_1 x_0),\ldots, x_s \cos(\lambda_s x_0),x_s \sin(\lambda_s x_0),x_{s+1}, \ldots, x_{n-1}),$$
where $\lambda_1, \ldots, \lambda_s$ are positive constants and
$s<n$. For the construction of the twisted cone, see
\cite{Baustere}.

\begin{theorem}[Bryant]\label{bryantfour}
  Let $V=\R^4$ and let
$\calS \subset S^2 V^*$ be a maximal austere subspace.  Then $\calS$
is $O(4)$-conjugate to one of the following:
\begin{align*}
(a) \qquad \Q_A &= \left\{\left.
\begin{bmatrix} A & B \\ B & -A \end{bmatrix}
\right| A,B \text{ are symmetric $2\times 2$ matrices} \right\},
\\
(b) \qquad \Q_B &= \left\{\left.
\begin{bmatrix} m I  & B \\ \transpose{B} & -m I \end{bmatrix}
\right| m \in \R,  B \text{ is a $2\times 2$ matrix}\right\},
\\
(c) \qquad \Q_C &= \left\{\left.
\begin{bmatrix} 0 & x_1 & x_2 & x_3 \\
    x_1 & 0 & \lambda_3 x_3 & \lambda_2 x_2 \\
    x_2 & \lambda_3 x_3 & 0 & \lambda_1 x_1 \\
    x_3 & \lambda_2 x_2 & \lambda_1 x_1 & 0
\end{bmatrix} \right| x_1, x_2, x_3 \in \R\right\},
\end{align*}
where, in the last case, parameters $\lambda_1,\lambda_2,\lambda_3$
satisfy
\begin{equation}\label{lambdaq}
\lambda_1 \ge \lambda_2 \ge 0 \ge\lambda_3, \qquad \lambda_1
\lambda_2 \lambda_3 + \lambda_1 + \lambda_2 + \lambda_3 = 0.
\end{equation}
\end{theorem}

We will say that an austere submanifold $M^4$ is of type A, B or C
respectively if for every point $p\in M$ the space $|\II_p|$ is
$O(4)$-conjugate to a subspace of the corresponding maximal austere
subspace given in Theorem \ref{bryantfour}. (In the case of type C,
we allow the parameters $\lambda_i$ to vary from point to point in
$M$.) It is possible for $|\II_p|$ to be conjugate to a subspace of
more than one maximal subspace (e.g., when $M$ is a hypersurface),
but we will assume that there is one particular maximal subspace
which applies at all points of $M$.

It is easy to give examples of austere 4-folds of types A and C.
When the codimension $r$ is even, any holomorphic submanifold $M^4
\subset \C^{2+(r/2)}$ is an austere $4$-fold of type A. To see this,
let $\rJ$ be the complex structure, and note that for any vector
fields $X,Y$ tangent to $M$
\begin{equation}\label{nuJnu}
\nu_a \cdot \II(X,\rJ Y) = \nu_a \cdot \nabla_X \rJ Y = \nu_a \cdot
\rJ (\nabla_X Y) = -(\rJ \nu_a) \cdot \nabla_X Y,
\end{equation}
and therefore
\begin{equation}\label{J2ff}
\nu_a \cdot \II(X,\rJ Y) = \nu_a\cdot \II(\rJ X, Y).
\end{equation}
In particular, if $X$ is an eigenvector for $\II$ in some normal
direction, then $\rJ X$ is an eigenvector for the opposite
eigenvalue. It also follows from \eqref{J2ff} that $\II$ is
represented by matrices in $\Q_A$ when we choose a moving frame
$e_1, e_2, e_3, e_4$ along $M$ such that   $\rJ e_1 = e_3$ and $\rJ e_2
= e_4$. As for austere 4-folds $M$ of type C, another result of
Bryant (see \cite{Baustere}, Theorem 3.1) implies that $M$ is a
generalized helicoid in $\R^7$ if and only if the parameters
$\lambda_1,\lambda_2,\lambda_3$ are identically zero.  On the other hand,
we do not know
if other austere 4-folds of type C exist. Similarly, the only
examples of austere 4-folds of type B which we know of are ruled by 2-planes.

\subsection*{Our Approach}
Our goals in studying austere submanifolds are to obtain new
examples and, where possible, to classify austere 4-folds of a given
type.  We employ the method of moving frames to generate exterior
differential systems (EDS) whose solutions correspond to austere
4-folds of a given type and codimension. When such systems are
involutive, Cartan-K\"ahler theory (see
\cite{BCG3}) gives us a measure of the size of
the solution space, in the form of what initial data may be chosen
for a sequence of Cauchy problems that determine every possible
local solution.  Studying the structure of the exterior differential
system can also enable us to establish global properties of
solutions (see, e.g.,  Prop. \ref  {Kprop1} and Prop. \ref {heliC} below).

One could organize a classification scheme for austere 4-folds in
Euclidean space by type and the dimension $\k$ of $|\II|$ (assumed
constant over the submanifold).\footnote{By duality, this $\k$
is also the rank of $\II$ as a linear map into the normal bundle; hence, 
we will often refer to it as the {\em normal rank} of $M$.}  However, we expect to obtain the
strongest classification theorems when the austere condition is
strongest, that is, when $\k$ is as large as possible for a given
type. Thus, like the earlier results of Bryant on 3-folds, the
classification results in this paper are obtained assuming that
$|\II|$ is conjugate to one of $\Q_A$, $\Q_B$ or $\Q_C$. (In this
case, we say $M$ is of {\em maximal} type A, B, or C, respectively.)
Classifying austere 4-folds of $M$ of non-maximal type would involve
parametrizing the possible subspaces of a given dimension
within $\Q_A$, $\Q_B$ or $\Q_C$ and analyzing the associated EDS.  In many
instances, the many additional parameters involved make the EDS
intractable, even with the assistance of computer algebra systems.

To obtain new examples, an often successful strategy is to assume
additional conditions.  In the last part of this paper, we obtain
new examples of austere 4-folds of non-maximal type A by assuming
that $\k=2$ and the space $|\II|$ lies on a non-principal orbit of
the action of the symmetry group of $\Q_A$ on the Grassmannian of
two-dimensional subspaces of $\Q_A$.   One can also carry out this
approach for type B with $\k=2$, but this yields no new examples.
The approach is not feasible for type C
because in that case the symmetry group is discrete.


\subsection*{Outline and Summary of Results}
In section \ref{framesec} we define the moving frames and associated
geometric structures we will use in the rest of the paper.  Because the
exterior differential systems we use are tailored for submanifolds
in a specific codimension, we prove a preliminary result in \S\ref{codsec}
to the effect that, when $|\II|$ satisfies certain algebraic criteria, 
then $\delta$ equals the {\em effective codimension} of $M$
(i.e., the codimension of $M$
within the smallest totally geodesic submanifold containing it).

Section \ref{Asection} is concerned with austere 4-folds of type A.
Before specializing to maximal type A, in \S\ref{typeAkahler}
we derive sufficient conditions,
in terms of $|\II|$, for a type A austere 4-fold $M$ to be K\"ahler.
(Note that when $M$ is type A, it carries a well-defined complex
structure $\rJ$ satisfying \eqref{J2ff}, but $\rJ$ need not extend
to the ambient space.)  This is a partial converse to a theorem of
Dajczer and Gromoll \cite{DG}, and we show that the K\"ahlerness conditions
apply whenever $|\II|$ has dimension at least two.
The main result of \S\ref{maxtypeA} is the
description of the generality of maximal type A austere
$4$-folds. We show that such
submanifolds in $\R^{r+4}$ depend on a choice of $2(r-1)$ functions
of 2 variables, in the sense of the Cartan-K\"ahler Theorem. 
We conclude that, generically, these austere
submanifolds of type A are not holomorphic submanifolds.

Section \ref{BCsection} is concerned with classifying
austere 4-folds of maximal types B and C.  In
\S\ref{Bsection} we show that austere 4-folds of maximal type B do not exist.
In section \S\ref{Csection}, we prove two results about maximal type C.
First, if $|\II |$ is at each point conjugate to a fixed
maximal austere subspace $\Q_C$ (i.e., the parameters $\lambda_i$ are assumed to be constant over $M$)
then $M$ must be a generalized helicoid.
Second, even without requiring the parameter values to be fixed,
we show that there is only a finite-dimensional family of submanifolds
of maximal type C.  This follows from showing that, away from certain exceptional parameter
values, the characteristic variety of the relevant EDS is empty; we also
show that the parameters take value in the exceptional locus on (at most)
the complement of an open dense subset of $M$.

In section \ref{examples} we give some interesting examples of
austere $4$-folds of non-maximal type.  As mentioned above,
one approach is to assume that $|\II|$, as a point in the 
Grassmannian of the relevant maximal austere subspace,
is non-generic for the action of the symmetry group
(i.e., it lies along a non-principal orbit).  
In \S\ref{typeAk2}, we use the symmetry group of $\Q_A$
to normalize 2-dimensional subspaces of $\Q_A$,
and identify the non-generic subspaces.   We then classify the
type A austere $4$-folds for which $|\II|$ has dimension two 
and is of fixed non-generic type, assuming that the Gauss map is nondegenerate.
(If the Gauss map of an austere 4-fold is degenerate, it must have
rank 2; such submanifolds 2 were classified by Dajczer and Florit \cite{DF}.)
These submanifolds, which all lie in a totally geodesic $\R^6$, turn out
to be either holomorphic submanifolds, products of minimal surfaces,
 or else $2$-ruled submanifolds.  The latter have the property that 
the image of the map $\gamma:M\to G(2,6)$,
taking point $p \in M$ into the subspace of $\R^6$ parallel to
the ruling through $p$, is a holomorphic curve.  Such curves are not arbitrary, however;
we also show how these ruled submanifolds may be constructed by instead choosing
a general holomorphic curve in $\CP^3$.

We plan to carry out a full classification of 2-ruled
austere 4-folds in our next paper.

\section{Moving Frames}\label{framesec}
In this section we discuss the first applications of moving frames to the geometry of austere
4-folds.  In particular, we obtain upper bounds on the effective codimension of the submanifold, and show that
adapted moving frames correspond to integrals of a certain Pfaffian exterior differential system on the appropriate frame bundle.

\subsection{Codimension}\label{codsec}
Let $M^4 \subset \R^{r+4}$ be austere.
For $p\in M$, let $N_p$ denote the orthogonal complement of $T_pM$ in $\R^{r+4}$.
A first approximation to the effective codimension of $M$ is the dimension of its
{\em first normal space} $N^1_p M$, which is the image of the second fundamental form $\II : S^2 T_p M \to N_p$.
Let $\k(p)$ denote the dimension of the first normal space,
which we will refer to as the {\em normal rank} of $M$ at $p$.
This is a lower semicontinuous function on $M$, bounded above by the codimension of $M$.  (For an austere submanifold of a given type, $\k(p)$ is also
bounded above by the dimension of the maximal austere subspace in which $|\II_p|$ lies.)
Hence $\delta(p)$ will be constant on an open set in $M$, so without loss of generality we will assume that $\k(p)$ is constant.


\begin{prop}\label{BCcodimension}
Suppose that $M \subset \R^{r+4}$ is type B or type C.  Then the effective codimension of $M$ equals $\k$.
\end{prop}
\begin{proof} Let $e_1,\ldots, e_4, \nu_1, \ldots, \nu_r$ be a moving frame along $M$, such that at each point $p$,
$e_1,\ldots, e_4$ span $T_pM$,
$\nu_1, \ldots, \nu_\k$ span $N^1_pM$,
and $\nu_{\delta+1},\ldots,\nu_r$ are orthogonal to
$T_pM \oplus N^1_pM$.

Let $S^a_{ij}(p) = \nu_a \cdot\II(e_i,e_j)$.
These symmetric matrices span the subspace $|\II_p|$,
when expressed in terms of the basis $e_1, \ldots, e_4$.
(We will use
index ranges $1 \le i,j,k\le 4$, $1 \le a,b,c\le \k$ and $\delta < \beta \le r$.)
Let
\begin{equation}\label{defofT}
\nabla_{e_i}\nu_a(p) =  T^\beta_{ai}\nu_\beta(p)
\mod T_pM \oplus N^1_p M.
\end{equation}
(We will use summation convention from now on.)  Differentiating
\begin{equation}\label{defofS}
\nabla_{e_j} e_i \equiv S^a_{ij} \nu_a  \mod T_pM
\end{equation}
along the $e_k$ direction, skew-symmetrizing in $j$ and $k$,
and taking the component in the direction of $\nu_\beta$ gives
\begin{equation}\label{TSeliot}
T^\beta_{ak}S^a_{ij} - T^\beta_{aj} S^a_{ik} = 0.
\end{equation}

Let $V=T_pM$ and let $\Q = |\II_p| \subset V^* \otimes V^*$.  Define
the {\em prolongation of $\Q$} as
\begin{equation}\label{Qprolongdef}
\Q^{(1)} := \Q \otimes V^* \,\cap\, V^* \otimes S^2 V^*.
\end{equation}
(This is a special case of the definition of the prolongation of a
subspace $\Q \subset W \otimes V^*$; see \cite{CFB}, Chapter 4.)
Then the equation \eqref{TSeliot} implies that for each $\beta$ the tensor
$U^\beta_{ijk} :=T^\beta_{ak}S^a_{ij}$ lies in the space $\Q^{(1)}$.

However, by Lemma \ref{BCprolong} below, the space $\Q^{(1)}$ has dimension zero.  If all the matrices $S^a_{ij}$ are identically
zero, then $M$ is totally geodesic and the proposition is true with
$\k = 0$.  Otherwise, the second fundamental form
is nonzero on an open set in $M$, and it follows that $T^\beta_{ak}=0$.  In that case, \eqref{defofT} and \eqref{defofS}
show that the span $\{e_1, \ldots e_4, \nu_1, \ldots, \nu_\k\}$
is fixed as we move along $M$.  Thus, $M$ lies in an affine linear
subspace of dimension $\k+4$.
\end{proof}

\begin{lemma}\label{BCprolong}
Let $\Q \subset \Q_B$ or $\Q\subset \Q_C$.
Then $\Q^{(1)}$, as defined by \eqref{Qprolongdef}, has dimension zero.
\end{lemma}

\begin{proof}  It suffices to verify that the prolongations of $\Q_B$ and
$\Q_C$ have dimension zero.

It is convenient for us to compute the prolongation as the space of
integral elements for a linear Pfaffian system with independence condition.
(See \cite{CFB}, Chapters 4-5, or \cite{BCG3}, Chapter 4 for more examples.)
In the case of $\Q_B$, the 2-forms of such a system would take the form
\begin{equation}\label{Btableau}
\begin{bmatrix}
\pi_0 & 0 & \pi_1 & \pi_2 \\
0 & \pi_0 & \pi_3 & \pi_4 \\
\pi_1 & \pi_3 & -\pi_0 & 0 \\
\pi_2 & \pi_4 & 0 & -\pi_0
\end{bmatrix} \& \begin{bmatrix} \w^1 \\ \w^2 \\ \w^3 \\ \w^4\end{bmatrix},
\end{equation}
where $\pi_0, \ldots, \pi_4, \w^1, \ldots \w^4$ are linearly independent 1-forms and
$\w^1 \& \w^2 \& \w^3 \& \w^4 \ne 0$ is the independence condition
for integral elements.  Thus, an integral element is described
by setting $\pi_a = P_{ak} \w^k$ for some coefficients $P_{ak}$ such that the above 2-forms vanish.
The possible values for these coefficients give a parametrization of the space $\Q_B^{(1)}$, since the prolongation is kernel of the composition
$$\Q_B \otimes V^* \longrightarrow S^2 V^* \otimes V^* \longrightarrow V^* \otimes \Lambda^2 V^*,$$
where the first map is inclusion and the second skew-symmetization.

The vanishing of the first 2-form in \eqref{Btableau}
implies that $\pi_0,\pi_1,\pi_2$ cannot contain any $\w^2$ terms.
Applying the same idea to each of the other 2-forms implies, in particular,
that on any integral element $\pi_0$ cannot contain terms involving $\w^1, \w^3$ or $\w^4$ either.  Thus $\pi_0 = 0$ on any integral element.
Then, examining the first and third rows in \eqref{Btableau} shows that, on any integral element, $\pi_1$ must lie in the intersection of spans $\{\w^1, \w^2\}$ and $\{\w^3, \w^4\}$, and thus must be zero.  We similar find the $\pi_2, \pi_3, \pi_4$ must vanish on any integral element, and the prolongation space
has dimension zero.

The proof that $\Q_C^{(1)}=0$ is similar.
\end{proof}

The argument of Lemma \ref{BCprolong} does not automatically apply to
subspaces of $\Q_A$, because the prolongation of $\Q_A$ is nonzero.
In fact, $\Q_A$ is easily seen to be the space of symmetric matrices that
anticommute with the complex structure represented by
\begin{equation}\label{defofJ}
J = \begin{bmatrix} 0 & 0 & -1 & 0 \\ 0 &  0 & 0 & -1 \\ 1 & 0 & 0 & 0\\ 0 & 1 & 0 &0\end{bmatrix}.
\end{equation}
Thus, because $-J = \transpose{J}$, a quadratic form ${\mathsf S}$ on the tangent space represented
by a matrix in $\Q_A$ is {\em $J$-linear}, i.e., ${\mathsf S}(X, JY) = {\mathsf S}(JX,Y)$.  The prolongation
of $\Q_A$ is the space of $J$-linear cubic forms, which is spanned by the real and imaginary parts of complex-linear cubic forms in two complex variables,
and thus has real dimension 8.

Likewise, the conclusions of Proposition \ref{BCcodimension} do not apply to all austere
4-folds of type A.  For, we may construct holomorphic submanifolds of real dimension 4
inside $\C^N$, for arbitrarily high $N$, which lie in no lower-dimensional subspace.
One example is the Segre embedding of the product of $\CP^1$ with a rational normal curve in $\CP^n$, given in terms of homogeneous coordinates $[u,v]$ and $[z,w]$ by
$$([u,v],[z,w]) \mapsto [z u^n, w u^n, z u^{n-1} v, w u^{n-1} v, \ldots, z v^n, w v^n].$$
This embedding maps $\CP^1 \times \CP^1$ into $\CP^{2n+1}$, and we obtain an
austere 4-fold in $\C^{2n+1}$ by intersecting with the domain of a standard chart
in projective space.

\subsection{Moving Frames and Pfaffian Systems}\label{standardsec}
Let $\F$ be the sub-bundle of the general linear frame bundle of $\R^{4+r}$ whose fiber
at a point $p$ consists of all bases $(e_1,\ldots,e_4,\nu_1,\ldots, \nu_r)$ for $T_p\R^{4+r}$ such that
the $e_i$ are orthonormal and orthogonal to the $\nu_a$.  (Here, we use
index ranges $1 \le i,j,k\le 4$ and $1 \le a,b,c\le r$.) We'll refer to $\F$ as the {\em semi-orthonormal} frame bundle.

As in \S\ref{codsec}, along a submanifold $M^4 \subset \R^{4+r}$
we may adapt a moving frame $e_1, \ldots, e_4$, $\nu_1, \ldots, \nu_r$ such that
at each $p\in M$ the frame vectors $e_1(p), \ldots, e_4(p)$ are an orthonormal basis for $T_p M$.
(The reason we do not also choose the $\nu_a$ to be orthonormal is that we will adapt them so
that the quadratic forms $\II_{\nu_a}$ are represented by a particular basis for an austere subspace.)
Then our moving frame along $M$ is a section of $\F\restr_M$.
 We will characterize such sections in terms of the canonical and connection 1-forms on $\F$.

These 1-forms are defined in terms of the exterior derivatives of the basepoint $p$ and
the frame vectors, regarded as $\R^{4+\k}$-valued functions on $\F$.  We let
\begin{equation}\label{denustreqs}
\begin{aligned}
dp &= e_i \w^i + \nu_a \theta^a,\\
de_i &= e_j \phi^j_i + \nu_a \eta^a_i,\\
d\nu_a &= e_j\xi^j_a + \nu_b \kappa^b_a,
\end{aligned}
\end{equation}
define the canonical forms $\w^i, \theta^a$ and connection forms
$\phi^i_j$, $\eta^a_i$, $\xi^i_a$ and $\kappa^a_b$.  These 1-forms span the cotangent
space of $\F$ at each point but are not linearly independent;
differentiating the equations $e_i \cdot e_j =\delta_{ij}$ and $e_i \cdot \nu_a =0$
yields the relations
\begin{equation}\label{xietarel}
\phi^i_j = -\phi^j_i, \qquad \xi^i_a = -\eta^b_i g_{ba},
\end{equation}
where $g_{ab} = \nu_a \cdot \nu_b$.
The exterior derivatives of the canonical forms satisfy structure
equations
\begin{equation}\label{blockstreqs}
d\begin{bmatrix} \omega \\ \theta \end{bmatrix} =
- \begin{bmatrix} \phi & -\transpose{\eta}\,g \\
\eta & \kappa \end{bmatrix} \wedge
\begin{bmatrix} \omega \\ \theta \end{bmatrix},
\end{equation}
where $\w,\theta,\phi,\eta,\kappa$ denote vector and matrix-valued 1-forms with
components $\w^i$, $\theta^a$, $\phi^i_j$, $\eta^a_i$, and $\kappa^a_b$ respectively,
and $g$ has entries $g_{ab}$.  Differentiating these equations, and noting that $\R^{4+\k}$ is flat, gives the
derivatives of the matrices of connection forms as
\begin{equation}\label{dconn}
\begin{aligned}
d\phi &= -\phi \wedge \phi +\transpose{\eta} \wedge g \eta,\\
d\eta &= -\eta \wedge \phi -\kappa \wedge \eta,\\
d\kappa &= \eta \wedge \transpose{\eta} g -\kappa \wedge \kappa,
\end{aligned}
\end{equation}
along with
$$d g= g\kappa + \transpose{\kappa} g.$$

We note the following fundamental fact relating adapted frames and submanifolds of $\F$:
\begin{quote}
A submanifold $\Sigma^4 \subset \F$
is a section given by an adapted
frame along some submanifold $M \subset \R^{4+\delta}$ if and only if
$\w^1 \& \w^2 \& \w^3 \& \w^4\restr_\Sigma \ne 0$ and $\theta^a\restr_\Sigma=0.$
\end{quote}
(We adopt the convention that $|_\Sigma$ for a differential form denotes the pullback to $\Sigma$ of that
form under the inclusion map.)
The first of these conditions is a non-degeneracy assumption called the {\em independence condition}.
The second condition implies, by differentiation, that
$$\eta^a_i \restr_\Sigma = S^a_{ij} \w^j.$$
for some functions $S^a_{ij}$.  These functions give the components of
the second fundamental form in this frame, i.e.,
\begin{equation}\label{noose}
\nu_a \cdot \II(e_i, e_j) = S^a_{ij}.
\end{equation}

\subsubsection*{The standard system}
We will now describe a class of exterior differential system (EDS), for later use, whose integral submanifolds are adapted frames along austere submanifolds.
(Being an integral submanifold of an EDS $\I$ means that the pullback to the submanifold
of any 1-forms in $\I$ is zero.)
To avoid tedious repetition, we will only consider integral submanifolds satisfying the above independence condition.
Unless otherwise stated, we will limit
our attention to austere submanifolds whose effective codimension $r$ equals the normal rank $\k$.

First, suppose we wish to construct an austere submanifold $M$ such that at each point $p$,
$|\II_p|$ is conjugate to a fixed austere subspace $\Q$ of dimension $\k$ (i.e, $M$ is of type $\Q$).
Let symmetric
matrices $\SH^1, \ldots, \SH^\k$ be a fixed basis for this subspace.  Then any such submanifold can be locally equipped with
an adapted frame such that \eqref{noose} holds.  Conversely, if submanifold $\Sigma^4 \subset \F$
is such that
$$\theta^a\restr_\Sigma=0, \qquad (\eta^a_i - \SH^a_{ij} \w^j)\restr_\Sigma = 0,$$
then it is the image of a section of
$\F\restr_M$ for some austere manifold $M$ of type $\Q$.  Thus, we may define on $\F$
a Pfaffian exterior differential system, the {\em standard system},
$$\I = \{ \theta^a, \eta^a_i - \SH^a_{ij}\w^j\}$$
whose integral submanifolds correspond to austere manifolds of this type.

We will also need to consider austere manifolds $M$ where $|\II|$ is conjugate
to an austere subspace $\Q_\lambda$ of fixed dimension $\k$ but which depends on parameters
$\lambda^1, \ldots, \lambda^\ell$ which are allowed to vary along $M$.
Suppose that a basis of this subspace is given by symmetric matrices
$S^1(\lambda), \ldots, S^\k(\lambda)$, and the parameters are allowed to range over an open set $L \subset \R^\ell$.
Then we may define the {\em standard system with parameters}
$$\I = \{ \theta^a, \eta^a_i - S^a_{ij}\w^j\},$$
which is analogous to the above, but now defined on the product $\F \times L$.  Given any austere manifold $M$ of this kind, we may construct an adapted frame along $M$
such that
$$\nu_a \cdot II(e_i,e_j) = S^a_{ij}(\lambda)$$
for functions $\lambda^1, \ldots, \lambda^\ell$ on $M$.  Then the
image of the fibered product of the mappings
$p \mapsto (p, e_i(p), \nu_a(p))$ and $p\mapsto (\lambda^1(p), \ldots, \lambda^\ell(p))$ will be an integral submanifold of $\I$.
Conversely, any integral submanifold of $\I$ satisfying the independence
condition gives (by projecting onto the first factor in $\F\times L$)
a section of $\F\restr_M$ which is an adapted frame for an austere manifold $M$.

For later use, we compute the 1-forms of $\I$.  We note that $d\theta^a\equiv 0$ modulo the
1-forms of $\I$, so that the only algebraic generator 2-forms are obtained from differentiating
the 1-forms $\theta^a_i := \eta^a_i - S^a_{ij}\w^j$.  Using \eqref{blockstreqs} and \eqref{dconn},
we obtain
\begin{equation}\label{gen2forms}
d\theta^a_i \equiv -(dS^a_{ij}-S^a_{kj}\phi_i^k-S^a_{ik}\phi_j^k+\kappa_b^aS^b_{ij})\wedge\omega^j
\end{equation}
modulo $\theta^a$ and $\theta^a_i$.  The 2-forms for the standard system without parameters are obtained
replacing $S^a$ in \eqref{gen2forms} with a constant $\SH^a$, for which $d\SH^a_{ij}=0$.

\section{Austere Submanifolds of Type A}\label{Asection}
In this section we classify maximal austere submanifolds of type A.
Of course, holomorphic submanifolds are of this type, but we
will show that general type A austere 4-folds are much more plentiful
than holomorphic
submanifolds.   Before specializing to submanifolds of maximal type A, we
will first characterize those type A
submanifolds for which the metric is K\"ahler.

\subsection{Real K\"ahler Submanifolds}\label{typeAkahler}
It was shown by Dajczer and Gromoll \cite{DG} that if a submanifold in
Euclidean space admits an almost complex structure for which the metric
inherited from the ambient space is K\"ahler, then the submanifold is austere.
(They termed these `real K\"ahler' submanifolds.)
As stated in \S\ref{codsec}, the maximal austere space $\Q_A$ may be characterized
as the set of symmetric matrices which anti-commute with a complex structure
on $\R^4$ represented by the matrix $J$ given by \eqref{defofJ}.  Thus,
the subgroup $U(2)^\R \subset GL(4,\R)$ of matrices that commute with $J$ preserve $\Q_A$.

 In general, we can
associate a well-defined almost complex structure\footnote{In fact, for a generic two-dimensional
 subspace $\Q \subset \Q_A$ the complex structure which anticommutes with matrices in $\Q$
 is uniquely defined up to a minus sign.} to type A austere 4-folds $M$.
Thus, it is natural to ask under what circumstances the metric on $M$ is K\"ahler.
Below, we will give a partial converse to Dajczer and Gromoll's result.  In order
to state our result precisely, we will need some algebraic preliminaries.

We split the space $\so(4)$ of skew-symmetric matrices as
$$\so(4) = \utwo^\R \oplus \calP,$$
where $\utwo^\R$ is the subspace of matrices that commute with $J$ (which
is isomorphic to the Lie algebra $\utwo$) and $\calP$ is the subspace of matrices which anticommute with $J$,
which is spanned by the matrices
$$
T = \begin{bmatrix}
0 & -1 & 0 & 0 \\
1 & 0 & 0 & 0 \\
0 & 0 & 0 & 1 \\
0 & 0 & -1 & 0
\end{bmatrix}, \quad
U = -J T = \begin{bmatrix}
0 & 0 & 0 & 1 \\
0 & 0 & -1 & 0 \\
0 & 1&  0 & 0 \\
-1 & 0 & 0 & 0
\end{bmatrix}.
$$

\begin{prop}\label{Kprop1}  Let $M$ be an austere 4-fold of type A, and for $p \in M$
let $|\II_p|$ be $O(4)$-conjugate to a subspace $\Q(p) \subset \Q_A$.  Define the map
$$\rK: S \mapsto \left([S,T] | [S,U]\right)$$
from the space of $4\times 4$ matrices $S$ into the space of $4\times 8$ matrices.
Suppose that for every $p$ the image of $\Q(p)$ under $\rK$
has the property that the common nullspace of all matrices in the image is zero.
Then $M$ is K\"ahler with respect to the
complex structure defined by $J$.
\end{prop}
\begin{proof}
Let $\SH^1, \ldots \SH^6$ be a fixed basis for $\Q_A$.  Let $r$ be the effective codimension of $M$,
not assumed to be the same as the normal rank of $M$.
Locally on $M$, we may
construct an adapted frame $e_1,\ldots, e_4, \nu_1, \ldots, \nu_r$ such that
$$\nu_a \cdot \II(e_i,e_j) = v^a_h \SH^h_{ij},$$
where $v^a_h$ are some functions on $M$ and $1\le h\le 6$.
Then the adapted frame defines a local section $f:M \to \F$ such that $f^* \w^i$ span the cotangent space of $M$
and the image of $f$ is an integral of the 1-forms $\theta^a$ and
$$\theta^a_i := \eta^a_i - v^a_h \SH^h_{ij}\w^j.$$

By specializing the computation \eqref{gen2forms} to the case where
$S^a_{ij}=v^a_h \SH^h_{ij}$, we obtain $d\theta^a_i \equiv -\Omega^a_i$ modulo the forms
$\theta^a_i$, where
\begin{equation}
\Omega^a_i :=\left((d v^a_h +\kappa^a_b v^b_h) \SH^h_{ij}-v^a_h [\SH^h,\phi]_{ij}\right) \&  \w^j
\label{minusOmega}
\end{equation}
and $[\SH^h,\phi]$ denotes the commutator.
These 2-forms must vanish under pullback via $f$.  Consider the 4-forms
$$\Xi^a_i := \Omega^a_i \& U_{k\ell}\, \w^k \& \w^\ell
+ \Omega^a_m \& J_{mi} T_{k\ell}\,\w^k \& \w^\ell.$$
Using the fact that $U=-J T$, we can expand these as
\begin{multline*}\Xi^a_i = -(d v^a_h +\kappa^a_b v^b_h) \&
(\SH^h_{ij}  (J T)_{k\ell}+J_{im}\SH^h_{mj} T_{k\ell} )\, \w^j \& \w^k \& \w^\ell\\
+v^a_h ([\SH^h,\phi]_{ij} (J T)_{k\ell}+J_{im}[\SH^h,\phi]_{mj}T_{k\ell}) \& \w^j \& \w^k \& \w^\ell.
\end{multline*}
Next, write $\phi = \phit + \psi$, where $\phit$ takes value in $\utwo^\R$ and
$\psi$ takes value in $\mathcal P$.  Using the fact that the matrices
$\SH^h$ and $[\SH^h,\phit]$ anticommute with $J$ while $[\SH^h,\psi]$ commutes with $J$, we have
\begin{multline*}
\Xi^a_i = \left(
(d v^a_h +\kappa^a_b v^b_h) \& \SH^h_{ij} (-\omega \& \transpose{\omega} J+J\omega \& \transpose{\omega})_{jk}
\right.
\\
+\left.v^a_h([\SH^h,\phit]_{ij} \& (-\omega \& \transpose{\omega} J+J\omega \& \transpose{\omega})_{jk}
+[\SH^h,\psi]_{ij} \& (\omega \& \transpose{\omega} J+J\omega \& \transpose{\omega})_{jk})\right)
\& T_{k\ell}\, \w^\ell.
\end{multline*}
It is easy to verify that $(-\omega \& \transpose{\omega} J+J\omega \& \transpose{\omega})\& T\omega=0$.
Computing the remaining terms gives
$$\Xi^a_i = -4 v^a_h [\SH^h,\psi]_{ij}U_{jk} \& \w_{(k)},
$$
where $\w_{(k)}$ denotes the 3-form which is the wedge product of the $\w^i$ such that
$\w^j \& \w_{(k)} = \delta^j_k\, \w^1 \& \w^2 \& \w^3 \& \w^4$.

Write $\psi = \psi_1 T + \psi_2 U$, where $\psi_1 = \tfrac12(\phi^2_1 - \phi^4_3)$
and $\psi_2=\tfrac12(\phi^3_2 - \phi^4_1)$.  Suppose that $f^*\psi_1 = a_k f^*\w^k$
and $f^*\psi_2 = b_k f^* \w^k$.  Then the vanishing of $\Xi^a_i$ and the
fact that $f^*(\w^1 \& \w^2\& \w^3 \& \w^4) \ne 0$, implies that the $a_k$ and $b_k$
must satisfy
$$v^a_h ([\SH^h, T]_{ij}U_{jk}a_k + [\SH^h,U]_{ij} U_{jk}b_k)=0.$$
Thus, if $\Q(p)$ satisfies the given conditions, then $a_k$ and $b_k$ must vanish.
Therefore, the connection forms $\phi$ of $M$ take value in $\utwo^\R$,
the connected component of the holonomy group of $M$ lies in $U(2)^\R$,
and it follows that $M$ is K\"ahler.
\end{proof}

Note that the vanishing of $\psi$ implies the vanishing of additional polynomials in
 the coefficients $v^a_h$.  To express these, introduce the
notation $\{ \}_\calP$ for the projection of an $\so(4)$-valued function
(or differential form)  into the subspace $\calP$.  By \eqref{dconn},
$$d\psi \equiv \{\transpose{\eta} \wedge g\eta\}_\calP\tmod \psi.$$
Furthermore, the $(i,j)$ entry of the matrix within braces is congruent,
modulo the $\theta^a_i$, to the 2-form $Y^i_{jk\ell}\w^k \& \w^\ell,$
where
$$Y^i_{jk\ell}:= g_{ab}v^a_h v^b_{h'} (\SH^h_{ik}  \SH^{h'}_{j\ell}-\SH^h_{i\ell}  \SH^{h'}_{jk}).$$
Then these additional conditions take the form
\begin{equation}\label{Kahlerintconds}
Y^2_{1k\ell} = Y^4_{3k\ell}, \qquad Y^3_{2k\ell} = Y^4_{1k\ell}
\end{equation}
for all $k <\ell$.

\begin{prop}\label{Kprop2} The only 2-dimensional subspaces of $\Q_A$ which do not satisfy
the hypothesis of Prop. \ref{Kprop1} are conjugate, via the action of $U(2)^\R$, to
the following:
$$\Q_x = \{ S_x, \Jt S_x \}, \quad \text{where} \
S_x = \begin{bmatrix} 1 & 0 & 0 & 0 \\ 0 & x & 0 & 0 \\ 0 & 0 & -1 & 0 \\ 0 & 0 & 0 & -x\end{bmatrix}, \quad \Jt =
\begin{bmatrix} 0 & 0 & 1 & 0 \\ 0 & 0 & 0 & -1 \\ 1 & 0 & 0 & 0 \\ 0 & -1 & 0 & 0\end{bmatrix}
$$
and $x$ is a real parameter.
\end{prop}
\begin{proof} Any nonzero matrix in $\Q_A$ can be diagonalized using $U(2)^\R$,
and scaled to be equal to $S_x$ for some $x \in \R$.  The kernel of $\rK(S_x)$
is a 4-dimensional subspace of $\R^8$.  Requiring that a general element $S\in \Q_A$, linearly independent
from $S_x$, has the property that the restriction of $\rK(S)$ to $\ker \rK(S_x)$
is singular, implies that $S$ must be a multiple of $\Jt S_x$.
\end{proof}

\begin{cor}\label{Kcorollary}
All type A submanifolds with $\k \ge 2$ are K\"ahler.
\end{cor}
\begin{proof} Because of Prop. \ref{Kprop1}, we need only check those
submanifolds such that $|\II_p|$ is at every point conjugate to a space of the form
$\Q_x$ defined in Prop. \ref{Kprop2}.  Note that matrices in $\Q_x$
anticommute with $\Jt$, and $\Jt = P J P^{-1}$ for a permutation matrix $P$.  Thus, we may repeat the argument of
the proof of Prop. \ref{Kprop1} with all matrices in $\so(4)$ replaced by their conjugates under $P$.
We conclude that the submanifold is K\"ahler with respect to the complex structure defined by $\Jt$.
\end{proof}

\subsection{Maximal Type A}\label{maxtypeA}
In the rest of this section, we discuss austere submanifolds of type A
with $\k=6$, i.e., whose second fundamental forms span the entire
space $\Q_A$.  We fix the following basis for this space:
\begin{align*}
\SH^1&=  \begin{bmatrix}
1 & 0 & 0 & 0 \\
0 & 0 & 0 & 0 \\
0 & 0 & -1 & 0 \\
0 & 0 & 0 & 0 \end{bmatrix}, &
\SH^2 &= \begin{bmatrix}
0 & 1 & 0 & 0 \\
1 & 0 & 0 & 0 \\
0 & 0&  0 & -1 \\
0 & 0 & -1 & 0
\end{bmatrix}, &
\SH^3 &= \begin{bmatrix}
0 & 0 & 0 & 0 \\
0 & 1 & 0 & 0 \\
0 & 0 & 0 & 0 \\
0 & 0 & 0 & -1
\end{bmatrix}, \\
\SH^4&=  \begin{bmatrix}
0 & 0 & 1 & 0 \\
0 & 0 & 0 & 0 \\
1 & 0 & 0 & 0 \\
0 &0& 0 & 0 \end{bmatrix}, &
\SH^5 &= \begin{bmatrix}
0 & 0 & 0 & 1 \\
0 & 0 & 1 & 0 \\
0 & 1&  0 & 0 \\
1 & 0 & 0 & 0
\end{bmatrix}, &
\SH^6 &= \begin{bmatrix}
0 & 0 & 0 & 0 \\
0 & 0 & 0 & 1 \\
0 & 0 & 0 & 0 \\
0 & 1 & 0 & 0
\end{bmatrix}, \\
\end{align*}

Let $M$ be an austere submanifold of this type; for the sake of simplicity, we
first assume that $\cod M = \k=6$, i.e., $M$ lies in $\R^{10}$.
Let $\F$ be the bundle of semi-orthonormal frames on $\R^{10}$,
as described \S\ref{standardsec}.
The manifold $\F$ has dimension $10+24+36+6$; referring to
the structure equations \eqref{blockstreqs}, we note that
the components of $\w$, $\theta$,
$\eta$, $\kappa$, and the lower triangle of $\phi$ give a coframe on $\F$.

By hypothesis, along $M$
there is a moving frame $(e_1,e_2,e_3,e_4,\nu_1, \ldots, \nu_6)$
such that
\begin{equation}\label{shatnu}
\nu_a \cdot \II(e_i,e_j) = \SH^a_{ij}.
\end{equation}
(Here, we take the convention that indices $i,j,k$ run between 1 and 4, while indices
$a,b$ run between 1 and $\k$.)  The image $\Sigma \subset \F$ is an integral
of the standard system $\I$ defined in \S\ref{standardsec}.
By \eqref{gen2forms}, we see that the 2-forms of $\I$ are given by $\pi^a_{ij}\&\w^j$ where
$$\pi^a_{ij} := \kappa^a_b \SH^b_{ij} - \SH^a_{ik}\phi^k_j - \SH^a_{jk} \phi^k_i.$$

Applying Propositions \ref{Kprop1} and \ref{Kprop2},
we see that $M$ must be K\"ahler; in particular,
the differential forms $\psi_1= \frac12(\phi^2_1 - \phi^4_3)$ and $\psi_2 = \frac12(\phi^3_2 - \phi^4_1)$
must also vanish along $\Sigma$.  Therefore, an adapted frame
along an austere 4-fold of this type will be an integral of the
augmented Pfaffian system
$$\I^+ =\{ \theta^a, \eta^a_i - \SH^a_{ij} \w^j, \psi_1, \psi_2 \}.$$
With the addition of the 1-forms $\psi_1,\psi_2$ come the additional integrability conditions \eqref{Kahlerintconds}.  In this case, these
generate only two linearly independent conditions,
\begin{equation}\label{gk6conds}
g_{13}-g_{22}-g_{46}+g_{55}=0, \qquad g_{16}-2g_{25}+g_{34}=0.
\end{equation}
Since these conditions are constraints on how the normal vectors $\nu_a$
may be arranged, they hold only on a codimension-two submanifold of the frame bundle
$\F$.  Let $\F'\subset \F$ denote this submanifold.  We now apply Cartan's
test for involutivity to the pullback of the Pfaffian system $\I^+$ to
$\F'$.

\begin{prop}\label{charA} On $\F'$, the system $\I^+$ is involutive with Cartan
characters $s_1=24$, $s_2=10$.
\end{prop}
\begin{proof}  As in the proof of Prop. \ref{Kprop1},
let $\phit$ be the projection of $\phi$ into $\utwo^\R$.
Then
$$d\theta^a_i \equiv \pit^a_{ij}\&\w^j \mod \theta^a, \theta^a_i,\psi_1, \psi_2,$$
where we define
$$\pit^a_{ij} := \kappa^a_b \SH^b_{ij} - \SH^a_{ik}\phit^k_j - \SH^a_{jk} \phit^k_i.$$
Next, let $\pit^a$ stand for the matrix-valued 1-form whose entries are $\pit^a_{ij}$.  On $\F$, the 36 1-forms $\kappa^a_b$ are linearly independent.  Because, for each $a$,  $\pit^a$ takes value
in the 6-dimensional space $\Q_A$, it follows that on $\F$ there are exactly
36 linearly independent forms among the $\pit^a_{ij}$.  Differentiating
\eqref{gk6conds} shows that two of these forms pull back to $\F'$ to be linearly
dependent on the others; for example, one can solve for
$\pi^4_{22}$ and $\pi^4_{24}$ in terms of the other $\pi^a_{ij}$.
It follows that, when pulled back to $\F'$, there are 24 linearly independent 1-forms
among the $\pi^a_{1j}$ and 10 further independent 1-forms among the $\pi^a_{2j}$.   This gives us the claimed values for the Cartan characters.

To apply Cartan's test, we need to calculate the fiber dimension of
the space of 4-dimensional integral elements (satisfying the independence condition) at points on $\F'$.
Suppose that such an integral element is defined by
$$\pit^a_{ij} = p^a_{ijk} \w^k$$
where $p^a_{ijk}$ is symmetric in $i,j,k$, and for any fixed $a$ and $k$ is in the space $\Q_A$.  For each $a$, the space of symmetric tensors satisfying
these conditions is isomorphic to the prolongation $\Q_A^{(1)}$, which has dimension 8.
As $a$ varies, we obtain a 48-dimensional
space of solutions $p^a_{ijk}$.  However, the corresponding integral 4-planes
must be tangent to the submanifold $\F'$.  This requirement imposes 4 additional
linearly independent homogeneous conditions on the $p^a_{ijk}$, so we conclude that the fiber dimension of the space of integral elements tangent to $\F'$
is 44.  Since this dimension coincides with $s_1+2s_2$, the system is involutive.
\end{proof}

We now state the following
\begin{theorem} Austere 4-folds in $\R^{10}$ of maximal type A exist and depend locally on a choice of 10 functions of 2 variables.
Each of them carries a complex structure with respect to which the metric inherited from ambient space is K\"ahler, but they are generically
not complex submanifolds.
\end{theorem}

\begin{proof}
The first assertion follows by applying the Cartan-K\"ahler Theorem
(cf. Theorem 7.3.3 in \cite{CFB}), given the fact that the system $\I^+$ is involutive
with characters computed in Prop. \ref{charA}.  The second assertion follows from Prop. \ref{Kprop1}.
The third assertion follows from the fact that
holomorphic submanifolds of real dimension four in $\R^{10}\simeq \C^5$ are locally the graphs of three holomorphic functions of
two complex variables.  The Cauchy-Riemann system for one function of two complex variables is involutive with Cartan
character $s_2=2$, so such submanifolds $M$ depend locally on a choice of 6 functions of 2 real variables.
But it is also instructive to see just how the solutions of the above Pfaffian system fail, in general, to be holomorphic submanifolds.

Suppose that $M^4 \subset \C^5$ is a holomorphic submanifold with normal space of real dimension 6,
and let $\rJ$ denote the ambient complex structure.  Adapt a framing along $M$ so that \eqref{shatnu} holds.
The equation \eqref{nuJnu} implies that $(\rJ \nu_a) \cdot \II(X,Y) = -\nu_a \cdot \II(X,\rJ Y )$.
It follows that the matrix representing $\rJ \nu_a \cdot \II$ must equal minus the product of the
matrix representing $\nu_a \cdot \II$ with $J$.  Multiplying the basis matrices $\SH^1,\ldots, \SH^6$
by $J$ shows that we must have
\begin{equation}\label{Jnoose}
\rJ \nu_1 = \nu_4, \qquad \rJ \nu_2 = \nu_5, \qquad \rJ \nu_3=\nu_6.
\end{equation}
By abuse of notation, we can assume that $\rJ$ is a constant matrix.  Differentiating, for example,
the equation $\rJ \nu_1 = \nu_4$, and using \eqref{denustreqs}, yields
$$\rJ(e_j \xi^j_1 + \nu_a \kappa^a_1) = e_j \xi^j_4 + \nu_a \kappa^a_4.$$
In particular, such framings satisfy $\kappa^1_1 = \kappa^4_4$, $\kappa^2_1=\kappa^5_4$
and $\kappa^3_1 = \kappa^6_4$.  More generally, differentiating the equations \eqref{Jnoose}
shows that the matrix $\kappa$ of 1-forms $\kappa^a_b$ must commute with the matrix
$$L = \begin{bmatrix} 0 & -I_{3\times 3} \\ I_{3\times 3} & 0 \end{bmatrix}.$$
Because $L^a_b \SH^b = \SH^a K$, we must have $\kappa^{a+3}_{b+3} = \kappa^a_b$ and $\kappa^{a+3}_b= -\kappa^a_{b+3}$
for $1 \le a,b\le 3$.  It follows that the 36 1-forms $\pit^a_{ij}$ must satisfy
$$L^a_b \pit^b_{ij} = \pit^a_{ik}K^k_j = \pit^a_{jk}K^k_i.$$
(In this equation, we revert to $1 \le a,b \le 6$.)
Because involutivity implies that integral manifolds may be constructed passing through
any given initial integral element, we see that a generic solution will not satisfy these extra necessary conditions.
\end{proof}

\def\abar{{\overline a}}
\def\bbar{{\overline b }}
As noted in \S\ref{codsec}, space $\Q_A^{(1)}$ has nonzero dimension,
so austere 4-folds of type A with $\k=6$ may in fact have codimension $r>6$.
To see how many of these there are, suppose that along such a submanifold $M$
we adapt moving frames, as in the proof of Prop. \ref{BCcodimension},
so that $\nu_1, \ldots, \nu_6$ span $N^1_pM$,
and $\nu_7,\ldots,\nu_r$ are orthogonal to $T_pM \oplus N^1_pM$.
(As before, let indices $a,b$ run from 1 to 6, but now let indices $\alpha,\beta$ run between $7$ and $r$.)

Such moving frames, as sections of $\F$, give integral submanifolds of
the following Pfaffian system:
$$\I^+ = \{ \theta^a, \theta^\beta, \psi_1, \psi_2, \eta^a_i - \SH^a_{ij}\w^j,
\eta^\beta_i \}.$$
Again, we restrict to the submanifold $\F'$ where the integrability conditions
\eqref{gk6conds} hold.  We compute
$$d\eta^\beta_i \equiv -\kappa^\beta_b \SH^b_{ij} \& \w^j$$
modulo the 1-forms in $\I^+$.  For every fixed index $\beta$,
the tableau component given by $\kappa^\beta_b \SH^b_{ij}$
is isomorphic to $\Q_A$, and is involutive with characters $s_1=4, s_2=2$.
Combining this with the results of Prop. \ref{charA} we conclude that the EDS $\I^+$ is involutive with characters $s_1 = 24+4(r-6)=4r$
and $s_2 = 10+2(r-6)=2r-2$.

We conclude that type A austere 4-folds in $\R^{4+r}$ with maximal first normal space (so that $r\ge 6$)
depend on a choice of $2(r-1)$ functions of 2 variables.  By contrast, when $r$ is
even, holomorphic submanifolds of real dimension 4 depend on $r$ functions of 2 variables.

\section{Maximal Types B and C}\label{BCsection}
\subsection{Submanifolds of Maximal Type B}\label{Bsection}
Let $M$ be an austere submanifold of type $B$ of normal rank $\delta$.
By hypothesis, there is a moving frame $(e_1, e_2, e_3, e_4,\nu_1, \ldots \nu_\delta)$
such that the $e_i$ are orthonormal and tangent to $M$, and
in each normal direction $\nu_a$ the shape operator takes the form
\begin{align}\label{mat}
\nu_a \cdot \II &= \begin{bmatrix} m^a I & B^a \\ \transpose{B^a} & -m^a I \end{bmatrix}
\end{align}
We consider the standard system with parameters
$${\mathcal I}=\{\theta^a,\ \eta^a_i-S^a_{ij}\omega^j\}$$ on ${\mathcal F}\times \R^{5\delta}$,
where ${\mathcal F}$ is the semi-orthonormal frame bundle of $M$ and $\nu_a\cdot II(e_i,e_j)=S^a_{ij}$
have form \eqref{mat}.  (For each $a$, the parameters are the scalar $m^a$ and
the entries of $B^a$.)  The integral submanifolds of this EDS with the  independence condition
$\omega_1\wedge\omega_2\wedge\omega_3\wedge\omega_4\not =0$ correspond to austere submanifolds of type B.
As in \eqref{gen2forms}, we compute the system $2$-forms as
$$d(\eta^a_i-S^a_{ij}\omega^j)\equiv
-(dS^a_{ij}-[S^a,\phi]_{ij}+\kappa_b^a S^b_{ij})\wedge\omega^j$$
modulo the $1$-forms of the ideal ${\mathcal I}$, where $[S^a,\phi]$ denotes the commutator.
Hence the tableau of the system is spanned by the $1$-forms
\begin {align}
\pi_{ij}^a:&=dS^a_{ij}-[S^a,\phi]_{ij}+\kappa_b^a S^b_{ij}\label{pij}
\end{align}

Let $\Omega_i^a=-\pi_{ij}^a\wedge \omega^j$ denote the system 2-forms,
which will also vanish on the integral elements.
The $5$-dimensional space of type B second fundamental forms ${\mathcal Q}_B$
is spanned symmetric matrices that anticommute with the reflection
$$R=\begin{bmatrix}
1 & 0 & 0 & 0 \\
0 & 1 & 0 & 0 \\
0 & 0 & -1 & 0 \\
0 & 0 & 0 &-1\end{bmatrix},$$
along with $R$ itself.
The space $\so(4)$ of skew-symmetric matrices splits as a direct sum of two subspaces
$W_1$ and $W_2$, where
$W_1= \left\{\left[\begin{smallmatrix} A_1 & 0 \\0 & A_2\end{smallmatrix}\right] , A_1, A_2\in \so (2)\right\}$
is the $2$-dimensional subspace of matrices which commute with $R$ and
$W_2=\left\{\left[\begin{smallmatrix}0 & B \\-\transpose{B}&0 \end{smallmatrix}\right], B\in {\mathcal M}_2(\R)\right\}$
is the $4$-dimensional subspace of matrices which anticommute with $R$.
We write the $\so(4)$-valued connection form $\phi$ as $\phi=\phi_0+\psi$,
where $\phi_0$ takes value in $W_1$ and $\psi$ takes value in $W_2$.

If $S\in {\mathcal Q}_B$, then $[S, \phi_0]$ anticommutes with the reflection $R$.
To see this, suppose $S$ belongs to the subspace of matrices in ${\mathcal Q}_B$ that
anticommute with $R$. Then
$$[S,\phi_0]R=S\phi_0R-\phi_0SR=SR\phi_0+\phi_0RS=-RS\phi_0+R\phi_0S=-R[S,\phi_0]$$
On the other hand, if $S$ is a multiple of $R$ we can see easily that $[S, \phi_0]$
again anticommutes with $R$.

The following result gives us the fiber dimension of the set of integral elements
of the Pfaffian system ${\mathcal I}$ with independence condition.
We point out that this result is independent of the normal rank of the submanifold.

\begin{prop}\label{sixteenelts}
The fiber dimension of the set of integral 4-planes (satisfying the
independence condition) of the Pfaffian system ${\mathcal I}$ is 16.
\end{prop}

\begin{proof}
The $2$-forms of the differential system ${\mathcal I}$ are given by
$\Omega^a_i:=-\pi_{ij}^a\wedge \omega^j$, where
$$\pi^a:=dS^a-[S^a,\phi]+\kappa_b^a S^b$$  It follows
that $\pi^a_{ij} = P^a_{ijk}\w^k$ on an integral element and the
integral elements are determined uniquely by the values of these $P^a_{ijk}\w^k$ . Then
\begin {align}
P^a_{ijk} &= P^a_{jik} \label{pes}
\end{align}
for every $a,i,j,k$.
Let $S^2 \R^4={\mathcal Q}_B\bigoplus {\mathcal U}$, where ${\mathcal U}$
is the orthogonal complement of ${\mathcal Q}_B$ in $V$ and it is a $5$-dimensional subspace.
Write $P^a_{ijk} = Q^a_{ijk} + R^a_{ijk}$, where $Q^a_{ijk}$ takes value in $\Q_B \otimes V$ and
$R^a_{ijk}$ takes value in $\mathcal U \otimes V$ for every index $a$.
Then the equations \eqref{pes}
are a linear system for the $Q^a_{ijk}$ and $R^a_{ijk}$.

Since $\phi=\phi_0+\psi$ and $[S^a,\phi_0]\in {\mathcal Q}_B$,
it follows that the projection of $\pi^a$ into the space ${\mathcal U}$
is the projection of $[S^a,\psi]$ onto ${\mathcal U}$.
Therefore, $R^a_{ijk}$ is completely determined by the value of
the $W_2$-valued 1-form  $\psi$ on the integral element.  Because $W_2$ is
4-dimensional, this depends on 16 parameters.  Now rewrite \eqref{pes} as a
non-homogeneous linear system for $Q^a_{ijk}$:
\begin{align}\label{qr}
Q^a_{ijk} - Q^a_{jik} &= - R^a_{ijk} + R^a_{jik}
\end{align}
The dimension of the solution space of this system is the same
for any set of values for the $R^a_{ijk}$.  In particular, when $R^a_{ijk}$
is zero, \eqref{qr} implies that $Q^a_{ijk}$ takes value in $\Q_B^{(1)}$, which
by Lemma \ref{BCprolong} is zero-dimensional.  Therefore, the values
of $Q^a_{ijk}$ satisfying \eqref{qr} are uniquely determined by the 16 parameters
that give the $R^a_{ijk}$.

\end{proof}
The next result shows that there are no type B austere $4$-folds
of maximal normal rank.

\begin{prop} A type B austere submanifold $M^4$ cannot have first normal space of dimension $\delta=5$.
\end{prop}
\begin{proof}
If $\delta=5$, then at each point the second fundamental form spans all of ${\mathcal Q}_B$.  We can therefore
choose smooth, linearly independent normal vector fields $\nu_1,\ldots,\nu_5$ so that $m^1,\ldots, m^4$ are identically zero, $B^1$ through $B^4$
 are given by
\begin{equation}\label{k4bvalues} B^1 = \begin{bmatrix} 1& 0 \\ 0
& 0\end{bmatrix}, B^2 =
\begin{bmatrix} 0 & 1\\ 0 & 0 \end{bmatrix}, B^3 = \begin{bmatrix}
0 & 0 \\ 1&0 \end{bmatrix}, B^4 = \begin{bmatrix} 0 & 0 \\ 0 &
1\end{bmatrix}.
\end{equation}
 and $m^5=1$, $B^5=0$.  This simplifies the formulas for the system 2-forms considerably and forces some additional $1$-forms, not in the ideal, to vanish on all admissible integral 4-planes.   Such 1-forms may be determined by examining the {\em tableau} of the
Pfaffian system.  For example, the system 2-forms $\Omega_i^5$ may be
 written in matrix-vector form as follows
 $$\begin{bmatrix} \Omega^5_1 \\ \Omega^5_2 \\ \Omega^5_3 \\ \Omega^5_4 \end{bmatrix}=
 \begin{bmatrix}
    \kappa^5_5  & 0     & \kappa^5_1+2\phi^3_1 & \kappa^5_2+2\phi_1^4 \\
    0                        & \kappa^5_5     & \kappa^5_3+2\phi_2^3 & \kappa^5_4 +2\phi^4_2 \\
    \kappa^5_1+2\phi^3_1 & \kappa^5_3+2\phi_2^3      & -\kappa^5_5                 & 0\\
    \kappa^5_2+2\phi_1^4  & \kappa^5_4 +2\phi^4_2& 0 & -\kappa^5_5
 \end{bmatrix} \& \begin{bmatrix} \w^1 \\ \w^2 \\ \w^3 \\ \w^4 \end{bmatrix} $$
This piece of the tableau allows us to deduce that $\kappa^5_5$ must vanish
on any admissible integral 4-plane $E$:  for, the vanishing of the 2-forms $\Omega_1^5$ and $\Omega _2^5$ implies that $\kappa^5_5$ restricts to $E$ to be a linear combination of $\omega^3$ and $\omega^4$, while the vanishing of  $\Omega_3^5$ and $\Omega _4^5$ implies that $\kappa^5_5$ restricts to be a linear combination of $\omega^1$ and $\omega^2$.
Thus,  $\kappa_5^5=0$ on any such integral element.  (See the proof of Prop. \ref{notone}
below for a more subtle example of this kind of calculation.)

In all, the additional 1-forms that vanish on all admissible integral elements are
$$
\begin{aligned}[t]
\psi_1 &= \kappa^5_1 + 2\phi^3_1,\\
\psi_2 &=\kappa^5_2 + 2\phi^4_1,\\
\psi_3 &=\kappa^5_3 + 2\phi^3_2,\\
\psi_4 &= \kappa^5_4 + 2\phi^4_2,\\
\psi_5 &=\kappa^5_5,\end{aligned}\qquad
\begin{aligned}[t]
\psi_6 &= \kappa^1_2 - \kappa^4_3 +2 \phi^4_3,\\
\psi_7 &=\kappa^1_3 -\kappa^4_2 +2\phi^2_1,\\
\psi_8 &= \kappa^2_1 -\kappa^3_4 -2\phi^4_3,\\
\psi_9 &=\kappa^3_1 - \kappa^2_4 -2\phi^2_1,\\
\psi_{10}&=\kappa^1_1 - \kappa^2_2 -\kappa^3_3 + \kappa^4_4,\\
\psi_{11}&=\kappa^1_4 +\kappa^2_3 + \kappa^3_2 +\kappa^4_1.
\end{aligned}
$$
Thus, any integral 4-fold of the EDS $\I$, satisfying the independence condition,
 will also be an integral of the 1-forms $\psi_1, \ldots,\psi_{11}$.  Let $\J$ be the
 differential ideal resulting from adding these 1-forms to $\I$.  The exterior derivatives
 of the $\psi$'s, modulo the 1-forms of $\J$, are linear combinations of
 wedge products of the $\kappa$'s with each other, and with the $\phi$'s.
 Thus, $\J$ is a non-linear Pfaffian system.  In particular, if we substitute the
 values given by $\pi^a_{ij}=P^a_{ijk}\omega^k$ into the new 2-forms, and take coefficients with respect to the
 2-forms $\w^1 \& \w^2$, $\w^1 \& \w^3$, $\w^1 \& \w^4$, $\w^2 \& \w^3$, $\w^2 \& \w^4$,
$\w^3 \& \w^4$, we obtain 66 quadratic polynomials in the $P^a_{ijk}$
which must vanish in order for an integral element of $\I$ to be an integral element of $\J$.
Eliminating the $P^a_{ijk}$ from these polynomials yields integrability
conditions in terms of the $g_{ab}$ which include $g_{11}+g_{55}=g_{22}+g_{44}=0$.
Since this is impossible for components of the metric,
we conclude that the set of integral 4-planes of $\J$ satisfying the independence condition
is empty.
\end{proof}

\subsection{Submanifolds of Maximal Type C}\label{Csection}

We begin by noting that the space ${\mathcal Q}_C$ of quadratic forms is preserved by
conjugation by a discrete subgroup of $O(4)$ that simultaneously permutes
$x_1,x_2,x_3$ and $\lambda_1, \lambda_2, \lambda_3$.  These permutations will, of course,
preserve the equation in \eqref{lambdaq} satisfied by the $\lambda$'s, but
will not preserve the inequalities in \eqref{lambdaq}.

We now discuss submanifolds of type C whose first normal space is of dimension $\delta=3$. These submanifolds lie in $\R^7$ as seen in Proposition $\ref{BCcodimension}$.
As was the case with submanifolds of type B whose second fundamental form had maximal span,
we can choose an orthonormal frame $e_1, e_2, e_3, e_4$ for the tangent space and a basis $\nu_1, \nu_2, \nu_3$ for the first normal space
with respect to which the second fundamental form is represented by any basis for the space ${\mathcal Q}_C$
we choose.  Accordingly, let $\F$ be the bundle of such frames on $\R^7$  and use the basis matrices
\begin{equation}\label{SchoiceC}
S^1=  \begin{bmatrix}
0 & 1 & 0 & 0 \\
1 & 0 & 0 & 0 \\
0 & 0 & 0 & \lambda_1 \\
0 &0&\lambda_1 & 0 \end{bmatrix}, \quad
S^2 = \begin{bmatrix}
0 & 0 & 1 & 0 \\
0 & 0 & 0 & \lambda_2 \\
1 & 0 & 0 & 0 \\
0 & \lambda_2 & 0 & 0
\end{bmatrix},
S^3 = \begin{bmatrix}
0 & 0 & 0 & 1 \\
0 & 0 & \lambda_3 & 0 \\
0 & \lambda_3 & 0 & 0 \\
1 & 0 & 0 & 0
\end{bmatrix},
\end{equation}
where we assume that
\begin{equation}\label{lambdarel}
\lambda_1 \lambda_2 \lambda_3 + \lambda_1 + \lambda_2 + \lambda_3 = 0.
\end{equation}
Let $1 \le i,j,k \le 4$ and $1 \le a,b \le 3$.  We recall the structure equations on $\F$ as
$$d\begin{bmatrix} \omega \\ \theta \end{bmatrix} =
- \begin{bmatrix} \phi & -\transpose{\eta}g \\
\eta & \kappa \end{bmatrix} \wedge \begin{bmatrix} \omega \\ \theta \end{bmatrix},
$$
where $\transpose{\w} = (\w^1, \w^2, \w^3, \w^4)$ is the vector of canonical forms corresponding
to the $e_i$ and $\transpose{\theta} = (\theta^1, \theta^2, \theta^3)$ is the vector of canonical
forms corresponding to the $\nu_a$, and $\phi$ is skew-symmetric.

Choosing a moving frame so that the second fundamental form in the
direction of $\nu_a$ is represented by matrix $S^a$ means that
$S^a$ gives the components of the 1-forms $\eta^a_i$ in terms of
the $\w^j$.    So, the moving frame will be an integral of the standard system with parameters:
\begin{align}
\I = \{\theta^a, \theta^a_i\} \qquad \theta^a_i :=
\eta^a_i -S^a_{ij} \w^j\label{2formsC}
\end{align}

 This Pfaffian system is defined on $\F \times L$, where $L \subset \R^3$ is the smooth affine
algebraic variety defined by \eqref{lambdarel}, minus the origin.
(We assume that not all the $\lambda$'s are zero, at least on an
open set in the submanifold; otherwise, the submanifold must be a
generalized helicoid.)  Using the permutation symmetry of $L$,
we may assume without loss of generality that $\lambda_1 \ge
\lambda_2 \ge 0$, and thus we may solve for $\lambda_3$ in terms of
$\lambda_1$ and $\lambda_2$.

The 2-forms of this EDS are given by \eqref{gen2forms}.
Let $\pi^a_{ij}=\pi^a_{ji}$ be the 1-forms defined by \eqref{pij}.
These 30 1-forms are not all independent; in fact, they are linear combinations
of the 17 independent 1-forms $d\lambda_1, d\lambda_2$, $\phi^i_j$ and $\kappa^a_b$.
The span of the components $\pi^a_{ij}$ within the cotangent space ${\mathcal F}\times L$ will be the same as the span of these 17 1-forms provided that $\lambda_2$ is nonzero.  (Otherwise, non-trivial Cauchy characteristics will be present.)
Let $L_0$ denote the open subset of $L$ where $\lambda_1 \ge \lambda_2 >0$.
At each point of $\F \times L_0$, the set of integral 4-planes of $\I$ satisfying
the independence condition has dimension 8, while the Cartan characters of the system
are $s_1 =12$ (for the 12 independent 1-forms $\pi^a_{1j}$, for example) and $s_2 =5$.
Since $8<s_1+2s_2$, the system fails to be involutive.

Without prolongation, we can obtain more information
about the system by calculating its characteristics:

\begin{prop} At points of $\F \times L_0$ where neither of
$\lambda_1$ or $\lambda_2$ is equal to 1, the characteristic variety of $\I$ is empty.
At points where $\lambda_1=1$ or $\lambda_2=1$, the characteristic variety consists of a pair of
complex lines.
\end{prop}
Consequently (using Thm. V.3.12 in \cite{BCG3}), the set of integral 4-folds which lie in the open subset satisfying $\lambda_1\ne 1$ and $\lambda_2\ne 1$
is at most finite-dimensional.
\begin{proof}  Let $E$ be an integral 4-plane annihilated by the $\pi^a_{ij}$.  (Because the
Pfaffian system $\I$ is linear, the characteristic variety is the same for every integral element at a point.)
Let  $\xi = \xi_i \w^i$ be a nonzero element of $E^*$, and let $\xi^\perp \subset E$ be the hyperplane
annihilated by $\xi$.  Then the polar equations of $\xi^\perp$ are generated by the
1-forms of $\I$ and the 1-forms $\pi^a_{ij}\xi_k - \pi^a_{ik}\xi_j$ for $j<k$.  By definition, the point $[\xi]\in \PP(E^* \otimes \C)$
is in the characteristic variety of $E$ if these equations fail to have full rank, i.e.,
the 72 1-forms $\pi^a_{ij}\xi_k - \pi^a_{ik}\xi_j$ have rank less than 17.

Expressing these 1-forms terms of the 1-forms $d\lambda_1, d\lambda_2$, $\phi^i_j$ and $\kappa^a_b$ yields a 72-by-17 matrix whose entries are linear functions of the $\xi_i$ with coefficients which are rational functions of $\lambda_1$ and $\lambda_2$.  We find that the matrix has full rank at points
where neither $\lambda_1$ nor $\lambda_2$ is equal to one, for any nonzero $\xi$.  On the other hand,
the matrix drops rank to 16 when $\lambda_1=1$ and $\xi$ lies on one of the two lines described by
$$ \xi_1 \xi_4 + \xi_2 \xi_3 = 0, \qquad \xi_3 = \pm \ri \xi_1.$$
Similarly, it drops rank to 16 when $\lambda_2 =1$ and $\xi$ lies on one of the lines given by
$$\xi_1 \xi_3 - \xi_2 \xi_4 = 0, \qquad \xi_2 = \pm \ri \xi_1.$$
\end{proof}

It turns out that in the case when  $\lambda_1=1$ or
$\lambda_2=1$, there are no integral submanifolds.

\begin{prop}\label{notone} A type $C$ austere submanifold $M^4$ cannot have normal rank $\delta=3$ and either $\lambda_1$ or
$\lambda_2$ identically equal to 1.
\end{prop}

\begin{proof}
Suppose $\lambda_1=1$. Equation \eqref{lambdarel} forces
$\lambda_2=-1$ or $\lambda_3=-1$. Without loss of generality, we take the case where $\lambda_3=-1$. Denote the remaining
parameter $\lambda_2$ by $\lambda$.

The Pfaffian system \eqref{2formsC} is defined on
${\mathcal F}\times \R$ and its $2$-forms are given by $\ref {gen2forms}$. The 1-forms $\ref {pij}$ are linear combinations of the 16 independent 1-forms $d\lambda,
\phi_i^j$ and $\kappa_a^b$.
The Cartan characters of ${\mathcal I}$ are computed
to be $s_1=12$ and $s_2=4$.
At each point of ${\mathcal F}\times \R$,
the set of integral $4$-planes of the differential system
${\mathcal I}$ with the independence condition
$\omega^1\wedge\omega^2\wedge\omega^3\wedge\omega^4\not= 0$ has
dimension 12.  Since $12 < s_1+2s_2$, the system
fails to be in involution. It turns out that there are four
additional $1$-forms that vanish on all integral $4$-planes and
should be added to the ideal \eqref{2formsC}. These are obtained
by studying the tableau $\pi_{ij}^a$ of the Pfaffian system.

For example, if we consider the first four lines of the tableau
(given by a=1), the $2$-forms obtained can be written in matrix
form as
$$\begin{bmatrix}\Omega_1\\ \Omega_2 \\ \Omega_3 \\ \Omega_4\end{bmatrix} :=\begin{bmatrix} 2\phi_2^1 & -\kappa_1^1 & \phi_1^4-\kappa_2^1-\phi_2^3 &
-\phi_2^4-\kappa_3^1+ \phi_1^3 \\
-\kappa_1^1 & -2\phi_1^2&
\phi_2^4+\kappa_3^1-\phi_1^3&\phi_2^3-\phi_1^4-\lambda \kappa_2^1\\
\phi_1^4-\kappa_2^1-\phi_2^3& \phi_2^4+\kappa_3^1-\phi_1^3& 2\phi_3^4 & -\kappa_1^1
\\-\phi_2^4-\kappa_3^1+\phi_1^3& \phi_2^3-\phi_1^4- \lambda
\kappa_2^1&-\kappa_1^1& -2
\phi_3^4\end{bmatrix}\wedge\begin{bmatrix}\omega^1\\
\omega^2\\ \omega^3\\ \omega^4\end{bmatrix}.$$
We calculate that
\begin{equation}\label{Cmissing}
\begin{aligned}
&\Omega_1\wedge\omega^3\wedge\omega^4+\Omega_3\wedge\omega^4\wedge\omega^1+\Omega_4\wedge\omega^2\wedge\omega^4=
\tfrac{1}{2}(\phi^2_1+\phi^4_3)\wedge\omega^1\wedge\omega^3\wedge\omega^4\\
&\Omega_1\wedge\omega^2\wedge\omega^3+\Omega_2\wedge\omega^4\wedge\omega^2+\Omega_3\wedge\omega^1\wedge\omega^2=
\tfrac{1}{2}(\phi^2_1+\phi^4_3)\wedge\omega^1\wedge\omega^2\wedge\omega^3\\
&\Omega_2\wedge\omega^3\wedge\omega^4+\Omega_3\wedge\omega^3\wedge\omega^1+\Omega_4\wedge\omega^2\wedge\omega^3=
\tfrac{1}{2}(\phi^2_1+\phi^4_3)\wedge\omega^3\wedge\omega^2\wedge\omega^4\\
&\Omega_1\wedge\omega^1\wedge\omega^3+\Omega_2\wedge\omega^4\wedge\omega^1+\Omega_4\wedge\omega^1\wedge\omega^2=
\tfrac{1}{2}(\phi^2_1+\phi^4_3)\wedge\omega^2\wedge\omega^1\wedge\omega^4.
\end{aligned}\end{equation}
Because the 2-forms $\Omega_i$ vanish on any admissible integral 4-plane $E$,
the same is true for the 4-forms on the right.  Then, since $\phi^2_1+\phi^4_3$ must
restrict to $E$ to be a linear combination of the $\w^i$, the simultaneous vanishing of 
the 4-forms on the right in \eqref{Cmissing} implies that this linear combination must be zero.

Similarly, one can use the other pieces of the tableau to show
that there are three more 1-forms that must vanish on the integral
elements. In all, these additional forms are
$$\begin{aligned}
\psi_1&=\phi^2_1+\phi^4_3\\
\psi_2&=\phi^4_1+\phi^3_2\end{aligned} \qquad\begin{aligned}
\psi_3&=\frac{4}{\lambda-1}\phi^4_1+\kappa^1_2\\
\psi_4&=\frac{4}{\lambda+1}\phi^2_1+\kappa^3_2,\end{aligned}
$$
where we now assume that $\lambda\not=1$ and $\lambda\not =-1$.
(We will consider the case where $\lambda =\pm1$ below.)
Let $\J$ be
the differential ideal obtained by adding the above four $1$-forms
to ${\mathcal I}$. This yields a non-linear Pfaffian system, since
the exterior derivatives of the new added forms will contain
linear combinations of wedges of the $\pi_{ij}^a$.
Computing $d\psi_1$ modulo the 1-forms of $\J$ gives
$$d\psi_1=-\lambda\omega^3\wedge\omega^4-\lambda\omega^1\wedge\omega^2.$$
Thus, integral elements satisfying the independence condition exist only
on the submanifold where $\lambda=0$.

We restricting  $\J$ to the submanifold where $\lambda=0$.
The integral $4$-planes satisfying the independence condition
will be defined by the equations
 \begin{align}
 \pi_{ij}^a&=s_{ijk}^a\omega^k\label{comb}
 \end{align}
 where now only 15 of the $1$-forms $\pi^a_{ij}$ are linearly independent,
 as linear combinations of $\phi_i^j$ and $\kappa_a^b$. Now we substitute the values in \eqref{comb}
into the new $2$-forms $d\psi_i, \ i=1..4$. For each of these, the coefficients with respect to
$\omega^i\wedge\omega^j$ for $i<j$ should all be zero. From these conditions, we get 12
quadratic polynomials in the $s_{ij}^a$ which must
vanish on any integral submanifold of $\J$. A Gr\"obner basis
calculation shows that these polynomials have no common zero, so
the set of $4$-integral elements of $\J$ is empty.

If $\lambda=1$ or $\lambda=-1$, the conclusion is the same. It
turns out that in this case there are 7 more 1-forms that vanish on any integral
element of $\I$ and which have to be added to the ideal.
Among the 1-forms of the augmented ideal $\J$ is $\phi^3_2$; when we compute the derivative
of this 1-form modulo the 1-forms of $\J$ we get
$$d\phi^3_2=2\omega^3\wedge\omega^2,$$
which can never vanish on admissible integral elements.
\end{proof}

The following result shows that in maximal codimension $\delta=3$ and when all parameters $\lambda_i$ are constant (i.e., the austere subspace does not vary from point to point), the $\lambda$'s are all forced to be equal to zero. This means that the second fundamental forms in various normal directions
are rank one, and contain a common linear factor; in Bryant's terminology, $M$ is called {\it simple}.
By Bryant's Theorem 3.1 in \cite{Baustere} it is congruent to a generalized helicoid.

\begin{prop}\label{heliC}
An austere $4$-fold $M$ of type C, with first normal space of dimension $\delta=3$
and such that the parameters $\lambda_1, \lambda_2, \lambda_3$ are constant,
must be a generalized helicoid.
\end{prop}

\begin{proof}
First, assume that none of the parameters $\lambda_i$ are zero.  Because of Prop. \ref{notone}, we can also
assume that none of them are equal to $\pm 1$.  We take the standard system $\I$ on $\F$ with
basis matrices $S^1,S^2,S^3$ given by \eqref{SchoiceC}, and calculate the system 2-forms
$\Omega^a_i := -\pi^a_{ij} \& \w^j$, where
$$\pi^a_{ij} = -[S^a,\phi]_{ij} + \kappa^a_b S^b_{ij}.$$
The components of $\pi^a_{ij}$ are linear combinations of the 15 linearly independent
forms $\phi^i_j$ and $\kappa^a_b$.  We claim that all of these forms must vanish on any admissible
integral element of $\I$.  For, substituting $\phi^i_j = s^i_{jk} \w^k$ and $\kappa^a_b = t^a_{bk}\w^k$
in the 2-forms, and equating the coefficients of the 6 2-forms $\w^i \& w^j$ to zero yields
a system of 72 homogeneous linear equations for the 60 variables $s^i_{jk}$ and $t^a_{bk}$.  By
a permutation of rows and columns, the matrix for this linear system is equivalent to one
with four nonzero $15\times 15$ blocks, each of which is nonsingular under our assumptions about
the values of the $\lambda_i$.  In particular, the connection forms $\phi^i_j$ must vanish
identically on any integral submanifold of $\I$, implying that the corresponding submanifold $M^4 \subset \R^7$
is totally geodesic.  This contradicts our assumption that $\delta=3$.

Next, we assume that exactly one of the parameters is identically zero.  Without loss of generality,
we may assume that $\lambda_3=0$ and $\lambda_2=-\lambda_1\ne 0$.
Then the fiber of the space of admissible integral elements of $\I$ has dimension two,
but then the following additional 1-forms vanish on all integral elements:
\begin{equation}\label{morceaux}
\phi^2_1-\lambda_1 \phi^4_3, \phi^3_1 + \lambda_1\phi^4_2,
\kappa^1_2+\kappa^2_1, \kappa^2_2-\kappa^1_1, \kappa^3_1-\phi^4_2, \kappa^3_2-\phi^4_3, \kappa^3_3.
\end{equation}
We adjoin these 1-forms to obtain a larger Pfaffian system $\J$.  However, taking the
exterior derivatives of the first two 1-forms in \eqref{morceaux} implies that
$g_{11}=g_{22}=0$, which is impossible for components of the metric on the normal bundle.

Thus, we conclude that the only possible solutions with parameters $\lambda_i$ all constant
are those for which all these parameters are zero.
\end{proof}

\section{Examples} \label{examples} In this section we examine some interesting examples of austere submanifolds whose normal rank $k$ is not maximal. More precisely, we describe the austere 4-folds of type A with $k=2$ and the austere hypersurfaces whose second fundamental form has rank two.
\subsection{Austere Submanifolds of Type A with $\k=2$}\label{typeAk2}

In this section we normalize the 2-dimensional subspaces of $\Q_A$ and classify the
corresponding austere 4-folds.  As stated in \S\ref{typeAkahler}, the symmetry group of $\Q_A$ is
$U(2)^\R = \{ M \in SO(4) | J M = M J \}$, with $J$ given by \eqref{defofJ},
and its action on $\Q_A$ is $M \cdot S = M S \itranspose{M}.$
This group is, of course, isomorphic to the group $U(2)$ of $2\times 2$
unitary matrices, and we can map $\Q_A$ to the space $\V$ of $2\times 2$
symmetric complex matrices in a way that preserves the action.  In particular,
if we define the maps
$$\rho : \begin{bmatrix} E & F \\ -F & E \end{bmatrix} \mapsto E + \ri F,
\qquad
\overline\rho : \begin{bmatrix} A & B \\ B & -A \end{bmatrix} \mapsto A - \ri B,$$
then $\overline\rho( M \cdot S ) = \rho(M) \cdot \overline\rho(S),$
where the action of $U(2)$ on $\V$ is again $U\cdot S = M S \itranspose{M}.$
In what follows, we will use this action to normalize real subspaces of
$\V$.

Let $\Q \subset \V$ be a subspace of real dimension 2, and let
$S,T$ span $\Q$.  We first consider the following special cases:

\underline{1. $S,T$ are linearly dependent over $\C$.}
In this case, we can use $U(2)$ to simultaneously diagonalize $S$ and $T$.
Using linear combinations with real coefficients, we can arrange that
\begin{equation}\label{STcaseI}
S = \begin{bmatrix} 1 & 0 \\ 0 & x+\ri y \end{bmatrix}, \quad T = \ri S,
\qquad x,y\in \R.
\end{equation}
We distinguish two subcases:
\newline
(1.a) every matrix in $\Q$ has full rank, so that
$x,y$ are not both zero; and
\newline
(1.b) the matrices $S$ and $T$ are singular (i.e., $x=y=0$).

\bigskip
\underline{2. $S,T$ are linearly independent over $\C$.}
We first note that, by the fundamental theorem of algebra, there
must be a singular matrix in the complex span of $S$ and $T$, i.e.,
\begin{equation}\label{detSTzero}
\det(T - \lambda S)=0.
\end{equation}
We distinguish several special cases:

\medskip
(2.a) \underline{$\Q$ contains a singular matrix} (i.e., $\lambda\in \R$).
In this case, we can linearly combine $S$ and $T$ so that $T$ has
rank 1.  Using $U(2)$, we can arrange that $\ker T$ is spanned
by $\stranspose{[1,0]}$; then using the diagonal subgroup $U(1)\times U(1)$
and real scale factors, we can assume that
\begin{equation}\label{STcaseIIa}
S = \begin{bmatrix} 1 & x+\ri y \\ x+\ri y & \ri u \end{bmatrix},\quad
T = \begin{bmatrix} 0 & 0 \\  0 & 1 \end{bmatrix}
\end{equation}
for real parameters $u,x,y$.

\medskip
(2.b) \underline{Quadrics in $\Q$ have a common nullspace.}
Assuming that the previous cases do not apply, in this case
 we can arrange that $S$ and $T$ have the form
\begin{equation}\label{STcaseIIb}
S = \begin{bmatrix} 0 & 1 \\ 1 & \ri x \end{bmatrix},\quad
T = \begin{bmatrix} 0 & \ri \\ \ri & y-x \end{bmatrix}
\end{equation}
for positive real parameters $x,y$. 

\medskip
(2.c) \underline{$S$ and $T$ commute.}  Assuming the previous
cases do not apply, in this case we can arrange that 
$$S = \begin{bmatrix} 1 & 0 \\ 0 & \ri y \end{bmatrix},\quad
T = \begin{bmatrix} \ri & 0 \\  0 & p-y \end{bmatrix}$$
for nonzero real parameters $p,y$.

\medskip
\noindent
Finally, we have the

(2.d) \underline{Generic case.}  When none of the above hold, 
we may arrange that 
$$S = \begin{bmatrix} 1 & u \\ u & x+\ri y \end{bmatrix}, \qquad
T = \begin{bmatrix} \ri & \ri u \\ \ri u & p-y + \ri x \end{bmatrix}
$$
for real parameters $p,u,x,y$ with $u,p$ nonzero.

(Further details for Case 2 are left to the interested reader.)

\begin{theorem}\label{k2specials}
Let $M \subset \R^{4+r}$ be an austere submanifold of type A,
with $\k=2$, such that $|\II|$ is of fixed nongeneric type on an
open dense subset of $M$, and the Gauss map of $M$ is nondegenerate.
Then $M$ lies in a totally geodesic $\R^6$.  Furthermore,
\newline
(i) if $|\II|$ is of type 1.a or 2.c, then $M$ is holomorphic
submanifold with respect to a complex structure on $\R^6$ given
by a constant matrix $\bigJ$;
\newline
(ii) if $|\II|$ is of type 2.a, then $M=\Sigma_1 \times \Sigma_2$ for  minimal surfaces $\Sigma_1, \Sigma_2$ in $\R^3$;
\newline
(iii) if $|\II|$ is of type 2.b, then $M$ is 2-ruled, and the image of the map $\gamma: M \to G(2,6)$
assigning to each point $p \in M$ the subspace of $\R^6$ parallel to the ruling through $p$ is a holomorphic curve .
\end{theorem}
Note that the Grassmannian $G(2,6)$ is endowed with a complex structure
that enables us to identify it with the standard quadric in $\CP^5$
(see \cite{KN}, Chapter XI, Example 10.6).
Note also that if the Gauss map of $M$ is degenerate, then it falls into case 1.b.
Austere manifolds with Gauss map rank 2, and $\k\ge 2$, were classified by Dajczer and Florit \cite{DF}.
For submanifolds of this type, the prolongation $|\II|^{(1)}$ is nonzero,
so we cannot conclude that they lie in a totally geodesic $\R^6$; in fact, the examples
of Dajczer and Florit exist in arbitrarily high effective codimension.

\begin{proof}For each case, let $S$ and $T$ be the normalized
basis matrices for the subspace, let $S^\R, T^\R$ denote their
inverse images under the map $\overline{\rho}$, and let $\Q$ be
the span of $S^\R$ and $T^\R$.  It is easy to check that $\Q^{(1)}=0$
in each case, so by the argument of Proposition \ref{BCcodimension},
$M$ lies in a totally geodesic $\R^6$.

\bigskip
(i) Assume that $|\II|$ is of type 1.a; then $\Q$ is parametrized
by $x,y$.  Taking $S^\R$ and $T^\R$ (where $S$ and $T$ are given
by \eqref{STcaseI}) as basis matrices, let $\I$ be the standard system with parameters,
 defined on $\F \times L$ where $L = \R^2$ minus the origin.
By Proposition \ref{Kprop2}, any austere manifold of this type will be
K\"ahler with respect to the complex structure given by \eqref{defofJ}.
Thus, the connection 1-forms must satisfy $\phi^2_1 = \phi^4_3$ and
$\phi^3_2 = \phi^4_1$ for any adapted frame that makes $|\II|$ conjugate to $\Q$.
Therefore, such adapted frames give integrals of the augmented system
\begin{equation}\label{defIplus}
\I^+=\{ \theta^1, \theta^2, \eta^1_i - S^\R_{ij} \w^j, \eta^2_i - T^\R_{ij}\w^j,
\phi^2_1 - \phi^4_3, \phi^3_2 - \phi^4_1\}.\end{equation}

Taking exterior derivatives of the last two 1-forms modulo the algebraic
ideal generated by forms in $\I^+$ shows that integral submanifolds satisfying
the independence condition exist only at points where
\begin{equation}\label{nugcond}
g_{11}=g_{22}, \qquad g_{12}=0.
\end{equation}
In other words, it is necessary that the frame vectors $\nu_1,\nu_2$ be
orthogonal and have the same length.  We pull back the system $\I^+$
to the submanifold $V \subset \F \times L$ where these conditions hold.
(Pulled back to $V$, the connection forms satisfy the additional relations
$\kappa^1_1=\kappa^2_2$ and $\kappa^2_1 = -\kappa^1_2.$)  On $V$,
the system is involutive with Cartan characters $s_1=4$, $s_2=2$.

To see that the corresponding austere submanifolds are holomorphic, we need to
endow $\R^6$ with the appropriate complex structure $\rJ$ which restricts
to $J$ on $M$.  Because $T^\R = S^\R J$, equation \eqref{nuJnu} implies that
this ambient complex structure must satisfy $\rJ \nu_1 = -\nu_2$.
Thus, if we let $F$ be a matrix whose columns are the vectors
$e_1, \ldots, e_4, \nu_1, \nu_2$, then $\rJ$ must satisfy
\begin{equation}\label{bigJeq}
\rJ F = F C, \qquad
C :=\left[\begin{array}{c|c}
J & 0 \\ \hline
0 & \begin{smallmatrix} 0 & 1 \\ -1 & 0\end{smallmatrix}\end{array}\right].
\end{equation}
If $\rJ$ is to be the standard complex structure on $\R^6$, then it
must be given by a constant matrix.  By \eqref{bigJeq}, this matrix
must equal $F C F^{-1}$.  Thus, we have only to show that, for any integral of $\I^+$, this matrix is a constant.

The structure equations \eqref{blockstreqs} imply that $dF = F \Phi$, where
$\Phi$ is the $6\times 6$ matrix of connection forms:
$$\Phi = \begin{bmatrix}\phi & -\stranspose{\eta} g \\ \eta & \kappa \end{bmatrix}.$$
Using this, we compute that $d\rJ= F [\Phi, C] F^{-1}$.  Then, it is easy
to verify that, for any integral of $\I^+ \restr_V$, the values of the
connection forms imply that $[\Phi,C]=0$.

The argument in the case that $|\II|$ is conjugate to a space of
type 2.c is similar, save that in that case $M$ is K\"ahler with respect
to the complex structure represented by $\Jt$.

\bigskip
(ii) We again set up the standard system with basis matrices $S^\R$ and $T^\R$
(where $S,T$ are given by \eqref{STcaseIIa}, and parameters $u,x,y$
range over all of $L=\R^3$).
The metric on $M$ is K\"ahler with respect to $J$, so we pass to the
augmented system $\I^+$ as in \eqref{defIplus}.  Again, differentiating
the last two 1-forms in $\I^+$ yield integrability conditions, which
in this case are
$$u=2xy, \qquad g_{12} = (x^2-y^2)g_{11}.$$
Let $V \subset \F\times L$ be the submanifold on which
these conditions hold.  The pullback of $\I^+$ to $V$ fails to be involutive,
In fact, integral elements exist only on the submanifold $V'$ defined
by $u=x=y=0$ and $g_{12}=0$.
The pullback of $\I^+$ to this submanifold
is involutive after one prolongation, with character $s_1=4$.

On $V'$, we compute (using the structure equations \eqref{denustreqs}) that
$$
\begin{aligned}d(e_1 \bigwedge e_3) &\equiv (\nu_1 \wedge e_3) \w^1 + (\nu_1 \wedge e_1) \w^3,\\
\nu_1 \bigwedge d\nu_1 &\equiv g_{11}\left( (\nu_1 \wedge e_3)\w^3 - (\nu_1 \wedge e_1) \w^1\right)
\end{aligned} \mod \I^+,
$$
where $\bigwedge$ is the exterior product on $\R^6$.  This shows that, for the framed
austere manifold corresponding to any solution of this EDS, the 3-plane through the
origin in $\R^6$ spanned by $e_1, e_3, \nu_1$ is fixed, and the orthogonal projection of $M$
onto this 3-plane is a rank 2 mapping.  The same is true for the 3-plane spanned by
$e_2, e_4, \nu_2$.  Thus, $M$ is the product of surfaces in these two copies of $\R^3$.
The austere condition implies that these must be minimal surfaces.

\bigskip
(iii) We set up the standard system with $S,T$ given by \eqref{STcaseIIb} for positive parameters
$x,y$; let $L \subset \R^2$ be the first quadrant.  The derivatives of the last
two 1-forms in $\I^+$ yield integrability conditions which are the same as \eqref{nugcond}.
The restriction of $\I^+$ to the submanifold $V \subset \F \times L$ where these
conditions hold is involutive, with character $s_1=8$.

To see that the corresponding austere manifolds are ruled, we compute the system 2-forms
$$
\begin{aligned}
d(\eta^1_1-\eta^2_3 - (S^\R_{1j}-T^\R_{3j})\w^j) &\equiv y(\phi^4_3 \& \w^2 - \phi^4_1 \& \w^4),\\
d(\eta^1_3+\eta^2_1 - (S^\R_{3j}+T^\R_{1j})\w^j) &\equiv y(\phi^4_1 \& \w^2 + \phi^4_3 \& \w^4)
\end{aligned}
\tmod \I^+_1,
$$
where $\I^+_1$ denotes the algebraic ideal generated by the 1-forms of $\I^+$.
It follows that on any solution there are functions $u_1, u_2$ such that
\begin{equation}\label{phivals}
\phi^2_1 = \phi^4_3 = u_1\w^2 + u_2 \w^4, \qquad
\phi^3_2 = \phi^4_1 = u_1 \w^4 - u_2 \w^2.
\end{equation}
Thus, we have
$$
\begin{aligned}d e_1 = e_3 \phi^3_1+(u_1 e_2+ u_2 e_4+ \nu_1) \w^2 + (u_1 e_4 - u_2 e_2+\nu_2) \w^4,\\
de_3 = e_1 \phi^1_3 + (u_1 e_4 - u_2 e_2+\nu_2) \w^2 -(u_1 e_2+ u_2 e_4+ \nu_1) \w^4
\end{aligned}.
$$
These equations show that, as we move along directions tangent to the $e_1$-$e_3$ plane in $T_p M$
(i.e., directions annihilated by $\w^2, \w^4$) the span of $e_1,e_3$ is fixed.  Thus, the map $\gamma:M \to G(2,6)$
has rank two.  To see that the image is a holomorphic curve, we must examine the complex structure on $G(2,6)$.

As in (\cite{KN}, {\it loc. cit.}), we think of $G(2,6)$ as $SO(6)/SO(2) \times SO(4)$.  Differential forms on $G(2,6)$---in particular, $(1,0)$-forms for
the complex structure---may be lifted up to the group $SO(6)$.  In this case, we use a lifting of the map $\gamma$
to a map $\Gamma:M \to SO(6)$ defined by
\begin{equation}\label{Gammadef}
\Gamma:p \mapsto (e_1(p),e_3(p),e_2(p),e_4(p),e_5(p),e_6(p)), \qquad e_5 = \tfrac1r \nu_1, e_6 = \tfrac1r \nu_2,
\end{equation}
where $r=\sqrt{g_{11}}$.  (Note the change in order, chosen so that the vectors tangent to the ruling at $p$ are the first two
columns of $\Gamma(p)$.)
Let $\Psi = G^{-1} dG$ be the Maurer-Cartan form on $SO(6)$, with components $\psi^i_j = -\psi^j_i$.
Then the forms $\psi^m_1, \psi^m_2$ for $3\le m \le 6$ are semibasic for the quotient map
$q:SO(6) \to G(2,6)$, which sends $G$ to the 2-plane spanned by its first two columns.
Moreover, the complex span of the 1-forms $\psi^m_1 -\ri \psi^m_2$ is well-defined on the quotient, and spans
the space of (1,0)-forms on $G(2,6)$.  Note also that by comparing the Maurer-Cartan equation $dG = G \Psi$ with
the defining properties \eqref{denustreqs} of the connection forms shows that for $1\le i,j \le 4$, we see that
\begin{equation}\label{translate}
\Gamma^*\psi^i_j = \phi^{\sigma(i)}_{\sigma(j)}, \qquad \Gamma^*\psi^5_j = r \eta^1_{\sigma(j)},
\quad \Gamma^*\psi^6_j = r \eta^2_{\sigma(j)},
\end{equation}
where $\sigma$ is the permutation that exchanges indices 2 and 3.

To show holomorphicity of $\gamma$, we have
to show that the pullback under $\gamma=q\circ \Gamma$ of the (1,0)-forms on $G(2,6)$
are (1,0)-forms on $M$, i.e., in the span of $\w^1 +\ri \w^3$ and $\w^2+\ri \w^4$.  (Note
that it is equivalent to show this for the pullbacks under $\Gamma$ for the
forms $\psi^m_1 -\ri \psi^m_2$.)  Using \eqref{phivals}
and \eqref{translate},
we compute
\begin{equation}\label{Gammapullbacks}
\begin{aligned}
\Gamma^*(\psi^3_1-\ri \psi^3_2) &= \phi^2_1-\ri \phi^2_3 &&= (u_1 - \ri u_2)(\w^2+\ri \w^4),\\
\Gamma^*(\psi^4_1 -\ri \psi^4_2)&=\phi^4_1-\ri \phi^4_3 &&= -(u_2+\ri u_1) (\w^2+\ri\w^4),\\
\Gamma^*(\psi^5_1 -\ri\psi^5_2) &= r (\eta^1_1 -\ri \eta^1_3) &&= r(\w^2+\ri \w^4),\\
\Gamma^*(\psi^6_1 -\ri\psi^6_2) &= r (\eta^2_1 -\ri \eta^2_3) &&= \ri r (\w^2+\ri\w^4).
\end{aligned}
\end{equation}
Notice that the holomorphic curve in $G(2,6)$ is not generic; indeed, if $M$ were
determined by specifying an arbitrary holomorphic curve in the Grassmannian as the image $\gamma(M)$,
then one would expect the Cartan character $s_1$ of $\I^+$ to be 6.
Instead, as we will see below, $M$ is in part determined by a holomorphic curve in a different Hermitian symmetric space.
\end{proof}

For the rest of this subsection, we will focus on 2-ruled austere submanifolds in $\R^6$ with $\k=2$,
the last type discussed in Theorem \ref{k2specials}.
As in the proof of that theorem, an adapted frame along such a submanifold $M$
defines the map $\Gamma:M \to SO(6)$ given by \eqref{Gammadef}.
Now let $\pi:SO(6) \to SO(6)/U(3)$ be the quotient map, where $U(3)$ is the intersection
of $SO(6)$ with
$$GL(3,\C)^\R = \{ M \in GL(6,\R) | M \bigJ = \bigJ M\}, \qquad
\bigJ := \left[\begin{smallmatrix}
0 & -1 & 0 & 0 & 0 & 0\\
1 & 0 & 0 & 0 & 0 & 0\\
0 & 0 & 0 & -1 & 0 & 0\\
0 & 0 & 1 & 0 & 0 & 0\\
0 & 0 & 0 & 0 & 0 & -1\\
0 & 0 & 0 & 0 & 1 & 0\end{smallmatrix}\right].
$$
On $SO(6)$, define the complex-valued 1-forms
$$
\begin{aligned}
\beta^1 &= \psi^4_1 - \ri \psi^4_2 +\ri(\psi^3_1 - \ri \psi^3_2),\\
\beta^2 &= \psi^6_1 -\ri\psi^6_2 + \ri(\psi^5_1 -\ri\psi^5_2),\\
\beta^3 &=\psi^6_3 -\ri \psi^6_4 + \ri(\psi^5_3 -\ri\psi^5_4)
\end{aligned}
$$
The quotient $SO(6)/U(3)$ has real dimension 6, and the space of semibasic forms
for the projection $\pi$ is spanned by the real and imaginary parts of the
$\beta^\ell$ for $1\le \ell \le 3$.
(These forms annihilate the left-invariant vector fields in the subalgebra $\mathfrak u(3)$, which
span tangent spaces of the fibres of $\pi$.)  The forms $\beta^i$ satisfy
\begin{equation}\label{betastreqs}
d\begin{bmatrix}\beta^1 \\ \beta^2 \\ \beta^3 \end{bmatrix}=
\Upsilon
\& \begin{bmatrix}\beta^1 \\ \beta^2 \\ \beta^3 \end{bmatrix},
\end{equation}
where
$$
\Upsilon:=\frac12
\begin{bmatrix}
2\ri(\psi^2_1+\psi^4_3) & \psi^5_3 -\ri \psi^5_4 +\psi^6_4+\ri \psi^6_3 &
            -(\psi^5_1-\ri \psi^5_2+\ri\psi^6_1+\psi^6_2) \\
-\psi^5_3 -\ri \psi^5_4 -\psi^6_4+\ri \psi^6_3 & 2\ri(\psi^2_1+\psi^6_5) &
            \psi^3_1-\ri \psi^3_2 +\ri\psi^4_1 +\psi^4_2\\
\psi^5_1+\ri \psi^5_2-\ri\psi^6_1+\psi^6_2 &
    -(\psi^3_1+\ri \psi^3_2 -\ri\psi^4_1 +\psi^4_2) &
        2\ri(\psi^4_3 + \psi^6_5)\end{bmatrix},
$$
indicating that the complex span of the $\beta^\ell$ is a pullback of a well-defined Pfaffian
system on $SO(6)/U(3)$, and these are the (1,0)-forms of an invariant (integrable) complex structure on the quotient.

It is evident from \eqref{Gammapullbacks} that
$\Gamma^*\beta^1=\Gamma^*\beta^2=0$.  Moreover,
\begin{equation}\label{gammastarbeta3}
\Gamma^*\beta^3 = r(\phi^6_2+\phi^5_4 +\ri(\phi^5_2-\phi^6_4)) \equiv
yr (\w^2 + \ri \w^4)\quad \tmod \I^+,
\end{equation}
indicating that the map $\pi \circ \Gamma:M \to SO(6)/U(3)$ has rank 2,
and is holomorphic.

In general, the space $SO(2n)/U(n)$ may be identified with
the set of orthogonal complex structures on $\R^{2n}$; in this case, with $n=3$,
it may also be identified with $\CP^3$, in the following way.\footnote{We learned
this identification in a paper of Abbena, Garbiero and Salamon \cite{Sal}.}
Let $V=\C^4$ with the standard hermitian metric, and
let $W=\R^6$ with the Euclidean metric.  Then $\Lambda^2 V = \C^6 = W\otimes \C$,
and each complex structure $\rJ$ on $W$ corresponds (by
associating to $\rJ$ its $+\ri$ eigenspace) to a totally isotropic subspace $E \subset W \times C$.
Each such subspace $E$ is of the form $E_u = \{ u \bigwedge v | v \in V\}$ for some
vector $u\in \C^4$.  The map $\rJ \mapsto E_u \mapsto u$ is well-defined up to
complex multiple, and identifies $SO(6)/U(3)$ with $\CP^3$.  Moreover,
the standard K\"ahler form on $\CP^3$ pulls back to
$\beta^1 \wedge \overline{\beta^1} + \beta^2 \wedge \overline{\beta^2} + \beta^3 \wedge \overline{\beta^3}$
on $SO(6)$, and $SO(6)$ may be identified with the unitary frame bundle of
$\CP^3$, with connection forms given by the components of $\Upsilon$.

The following result shows that the association of $M$ with a holomorphic
curve in $\CP^3$ is surjective but not 1-to-1:

\begin{theorem}  Let $\CC$ be a holomorphic curve in $\CP^3$.  Given a non-planar point $p\in \CC$,
there is an open neighborhood $U \subset \CC$ containing $p$ and a 2-ruled austere
manifold $M \subset \R^6$ such that $\pi\circ \Gamma(M)=U$.  Such manifolds $M$
depend on a choice of 4 functions of 1 variable.
\end{theorem}

\begin{proof}The proof of Theorem \ref{k2specials} part (iii) shows that
along $M$ there is an orthonormal frame $(e_1, \ldots, e_6)$ such that
$e_5 \cdot \II(e_i,e_j)=S^1_{ij}$ and $e_6 \cdot \II(e_i, e_i)=S^2_{ij}$ for
matrices
$$S^1 = r\begin{bmatrix}0 & 1 & 0 & 0\\ 1 & 0 & 0 & x\\ 0 & 0 & 0 &-1\\0 & x & -1 & 0\end{bmatrix},\qquad
S^2 = r\begin{bmatrix} 0 & 0 & 0 & 1 \\ 0 & y-x & 1 & 0\\ 0 & 1 & 0 & 0\\ 1& 0 & 0 & x-y\end{bmatrix}
$$
and some positive functions $r,x,y$ along $M$.
To construct $M$, we will set up a Pfaffian system, similar to the augmented standard system $\I^+$, satisfied by the orthonormal frame.

Let $\F_o$ be the orthonormal frame bundle of $\R^6$; in terms of the
bundle $\F$ of semi-orthonormal frames defined in \S\ref{standardsec},
$\F_o$ is the sub-bundle of $\F$ on which
\begin{equation}\label{onspecial}
g_{11}=g_{22}=1,\qquad g_{12}=0
\end{equation}
hold.
The structure equations of $\F_o$ are the same as those given by equations
\eqref{denustreqs} through \eqref{blockstreqs}, but with the specialization
\eqref{onspecial} taken into account, $\kappa$ is skew-symmetric and
$\xi = -\stranspose{\eta}$.

We adjoin $r,x,y$ as new variables, taking value in the positive octant $L\subset \R^3$, and define our Pfaffian system $\J$ on $\F_o \times L$ to
be generated by
$$\theta^1, \theta^2, \eta^a_i - S^a_{ij} \w^j, \ \phi^4_3-\phi^2_1,\  \phi^4_1-\phi^3_2.$$
As in \eqref{Gammadef}, we define a map $\Gamma:\F_o \times L \to SO(6)$, whose
value is the matrix with columns $(e_1, e_2, e_3, e_4, \nu_1, \nu_2)$.
We will now show how, given an arbitrary holomorphic curve $\CC \subset \CP^3$,
we construct an integral manifold of $\J$ whose image, under $\pi\circ \Gamma$,
is an open neighborhood of $p\in \CC$.

Let $N = \pi^{-1}(\CC)$.  On $N$, the complex span of the 1-forms $\beta^\ell$ is one-dimensional at each point.  If
$z$ is a local holomorphic coordinate on $\CC$ near $p$,
then there will be complex functions $f^\ell$ such that $\beta^\ell = f^\ell\pi^* dz$ on $N$.
Substituting this in \eqref{betastreqs} gives
\begin{equation}\label{dief}
d f^\ell \equiv \Upsilon^\ell_m f^m \tmod dz.
\end{equation}
Let $F$ be a fiber of $\pi:N \to \CP^3$.  Since such fibers are left
cosets of $U(3) \subset SO(6)$, $U(3)$ acts simply transitively
on them by right multiplication.  This action is generated infinitesimally by the left-invariant vector
fields on $SO(6)$ that are tangent to subalgebra $\mathfrak u(3)$ at the identity.  Thus, such vector
fields give a frame field tangent to $F$.  Because $\Upsilon$ is $\mathfrak u(3)$-valued, \eqref{dief} shows that
as the action of $U(3)$ moves points along $F$,
the corresponding action on the vector with components $f^\ell$ is isomorphic to the standard action of $U(3)$ on $\C^3$.
Thus, in each fiber there is a subset where $f^1=f^2=0$, and
the union of these subsets is a smooth submanifold $N' \subset N$ of codimension 4 within $N$.
(Note that the definition of $N'$ does not depend on the choice of local coordinate on $\CC$.)

Because $\beta^1=\beta^2=0$, then $\Upsilon^1_3=-\psi^5_1+\ri\psi^5_2$ and
$\Upsilon^2_3 = \psi^3_1 - \ri\psi^3_2$.  Thus,
the restriction of \eqref{betastreqs} to $N'$ implies that there
are complex functions $f,k$ on $N'$ such that
\begin{equation}\label{fksemi}
\psi^3_1 - \ri\psi^3_2 = f \beta^3, \qquad \psi^5_1-\ri\psi^5_2 = k \beta^3.
\end{equation}
Besides $\beta^1=\beta^2=0$, these are the
only linearly dependencies among the left-invariant 1-forms of $SO(6)$
when restricted to $N'$; thus, the 1-forms
$\psi^2_1, \psi^4_3, \psi^5_3, \psi^5_4, \psi^6_3,\psi^6_4,\psi^6_5$
(which include the real and imaginary parts of $\beta^3$) form a coframe on $N'$.

Computing
$$\Gamma^*(\psi^5_1 - \ri \psi^5_2) = \eta^1_1 -\ri \eta^1_3 \equiv
r(\w^2+\ri \w^4) \mod \J$$
and comparing with \eqref{gammastarbeta3} shows that
on the image under $\Gamma$ of a solution of $\J$, we must have
$k$ equal to $1/y$.  Thus, we need to restrict to the
subset $N'' \subset N'$ where $k$ is real and positive.
Differentiating \eqref{fksemi} gives
\begin{equation}\label{deek}
dk = \ri k(\psi^2_1 -\psi^4_3 -2\psi^6_5 ) -f(\psi^6_4 +\ri\psi^6_3)
+ w\beta^3
\end{equation}
for some complex function $w$ on $N'$.
This equation shows that the
subgroup of $U(3)$ stabilizing $N'$ can be used to make $k$ real and positive,
provided that $f$ and $k$ are not both identically zero along a fiber.
Since $f,k$ give the components of the second fundamental form of $\CC$
as a holomorphic submanifold of $\CP^3$, then $f=k=0$ along the fiber above $p$
means that $p$ is a planar point of $\CC$.
Thus, we will restrict to non-planar points of $\CC$.
Then, in each fiber of $N'$ there is a subset where $k>0$, and the
union $N''$ of these subsets is a smooth codimension-one submanifold of $N'$.

If we let $f = g+\ri h$ and $w=u+\ri v$ for real functions $g,h,u,v$, then taking the real and imaginary
parts of \eqref{deek} gives
\begin{equation}\label{dkval}
dk = -v\psi^5_3 +u\psi^5_4 +(h+u)\psi^6_3 +(g-v)\psi^6_4
\end{equation}
and $k(\psi^2_1 -\psi^4_3 -2\psi^6_5 ) +u\psi^5_3 + v\psi^5_4 -(g-v)\psi^6_3 -(h+u)\psi^6_4=0$
on $N''$.  Because of the last equation, we may use the
restrictions of the 1-forms $\psi^4_3, \psi^5_1, \psi^5_2, \psi^5_3, \psi^5_4, \psi^6_5$
as a coframe on $N''$.

Let $\Sigma$ be the smooth hypersurface in $\Gamma^{-1}(N'')$ defined
by $y=1/k$. Because the fibers of $\Gamma$ have dimension 9,
$\Sigma$ has dimension 14, with coframe given by
$\w^1, \ldots, \w^4$, $\theta^1, \theta^2$,
$\phi^4_2$, $\eta^1_1$, $\eta^1_2$, $\eta^1_3$, $\eta^1_4$, $\kappa^2_1$, $dx$ and $dr$.
From \eqref{dkval} we deduce that
\begin{equation}\label{dyval}
dy = y^2(g \eta^1_2 + h \eta^1_4) -y^3((u+h)\eta^1_1 + (v-g)\eta^1_3)
\end{equation}
on $\Sigma$.
The 1-forms in $\J$ pull back to $\Sigma$ to give a rank 6 system generated by
$\theta^1, \theta^2$ and
$$
\begin{aligned}
\alpha^1 &= \eta^1_1 - r\w^2,\\
\alpha^2 &= \eta^1_2 + r\w^4,\\
\alpha^3 &= \eta^1_3 - r(\w^1 + x\w^4),\\
\alpha^4 &= \eta^1_4 + r(\w^3-x\w^2).
\end{aligned}
$$

On $\Sigma$, there are 1-forms $\pi_1, \pi_2, \pi_3, \pi_4$
which are linearly independent combinations of $\phi^4_2, \kappa^2_1, dx, dr$ modulo $\w^1, \ldots,\w^4$, such that
$$
d\begin{bmatrix}\alpha^1 \\ \alpha^2 \\ \alpha^3 \\ \alpha^4 \end{bmatrix}
\equiv
\begin{bmatrix} 0 & \pi_1 & 0 & \pi_2 \\ 0 & \pi_2 & 0 & -\pi_1 \\
\pi_1 & (y-2x)\pi_4  & \pi_2 & \pi_3  \\
\pi_2 & \pi_3  & -\pi_1 & -(y-2x) \pi_4
\end{bmatrix}
\& \begin{bmatrix}\w^1 \\ \w^2 \\ \w^3 \\ \w^4 \end{bmatrix}
\tmod \alpha^1, \ldots, \alpha^4.
$$
This indicates that, at points where $x\ne y/2$, the system $\J\restr_\Sigma$
is involutive with terminal Cartan character $s_1=4$.
Local existence of solutions then follows by applying the Cartan-K\"ahler Theorem at such points.
\end{proof}


\begin{thebibliography}{10}
\bibitem{Sal} Abbena, E., Garbiero, S., Salamon, S.,
{\sf Almost hermitian geometry in six-dimensional nilmanifolds},
Ann. Norm. Sup. Pisa (Scienze, 4th series) {\bf 30} (2001), 147--170.

\bibitem{Baustere}
Bryant, R.L., {\sf Some remarks on the geometry of austere manifolds},
Bol. Soc. Bras. Mat. {\bf 21} (1991), 133--157.

\bibitem{BCG3}
Bryant, R.L., Chern, S.-S., Gardner, R.B., Goldschmidt, H.L., Griffiths, P.A.,
{\it Exterior differential systems}, MSRI Publications {\bf  18}, Springer-Verlag, New York, 1989.

\bibitem{DF}
Dajczer, M., Florit, L., {\sf A class of austere submanifolds},
Illinois J. Math. {\bf 45} (2001), 735--755.

\bibitem{DG}
Dajczer, M., Gromoll, D., {\sf Gauss parametrizations and Rigidity Aspects of Submanifolds},
J. Diff. Geom. {\bf 22} (1985), 1--12.

\bibitem{HL}
Harvey, R., Lawson, H.B., {\sf Calibrated Geometries}, Acta Math. {\bf 148} (1982), 47--157.

\bibitem{CFB}
Ivey, T.A., Landsberg, J.M.,
{\it Cartan for Beginners: Differential geometry via moving frames and
exterior differential systems}, Graduate Studies in Mathematics {\bf 61}, American Mathematical Society,
Providence, 2003.

\bibitem{KM} S. Karigiannis, M. Min-Oo,
{\sf Calibrated subbundles in noncompact manifolds of special holonomy},
Ann. Global Anal. Geom.  {\bf 28}  (2005), 371--394.

\bibitem{KN} S. Kobayashi, K. Nomizu, {\it Foundations of Differential Geometry},
Wiley-Interscience, 1969.

\bibitem{St} M. Stenzel, {\sf Ricci-flat metrics on the complexification of a compact rank one symmetric space},  Manuscripta Math. 80  (1993), 151--163.
\bibitem{SYZ} A. Strominger, S.-T. Yau, and E. Zaslow, {\em Mirror
symmetry is $T$-duality}, Nuclear Physics B {\bf 479} (1996), 243--259.

\end{thebibliography}
\end{document}